\documentclass[12pt,reqno]{amsart}

\usepackage{multirow}
\usepackage{hhline}

\usepackage{mathtools}
\usepackage{times}
\usepackage[T1]{fontenc}
\usepackage{mathrsfs}
\usepackage{latexsym}
\usepackage[dvips]{graphics}
\usepackage[titletoc, title]{appendix}
\usepackage{amsmath,amsfonts,amsthm,amssymb,amscd}
\usepackage[dvipsnames]{xcolor}
\usepackage{hyperref}
\usepackage{amsmath}
\usepackage[shortlabels]{enumitem}

\usepackage{color}
\usepackage{breakurl}

\usepackage{comment}

\newtheorem{thm}{Theorem}[section]

\newtheorem{claim}[thm]{Claim}
\newtheorem{lem}[thm]{Lemma}
\newtheorem{prop}[thm]{Proposition}

\newtheorem{defi}[thm]{Definition}
\newtheorem{rek}[thm]{Remark}

\DeclareMathOperator{\spann}{span}
\DeclareMathOperator{\sgn}{sgn}
\numberwithin{equation}{section}

\DeclareFontFamily{U}{mathx}{}
\DeclareFontShape{U}{mathx}{m}{n}{<-> mathx10}{}
\DeclareSymbolFont{mathx}{U}{mathx}{m}{n}
\DeclareMathAccent{\widehat}{0}{mathx}{"70}
\DeclareMathAccent{\widecheck}{0}{mathx}{"71}

\newcommand{\BB}{\mathcal{B}}

\newcommand{\yy}{\mathbf{y}}

\newcommand{\NN}{\mathbb{N}}

\newcommand{\cF}{\mathcal{F}}

\newcommand{\cY}{\mathcal{Y}}

\newcommand{\floor}[1]{\left\lfloor #1 \right\rfloor}
\newcommand{\ceil}[1]{\left\lceil #1 \right\rceil}

\begin{document}

\title[]{On sequential greedy-type bases}

\author{Miguel Berasategui}
\address{M. Berasategui, IMAS - UBA - CONICET - Pab I, Facultad de Ciencias
Exactas y Naturales, Universidad de Buenos Aires, (1428), Buenos Aires, Argentina}
\email{mberasategui@dm.uba.ar}

\author{Pablo M. Bern\'a}
\address{P. M. Bern\'a, Departamento de Matemáticas, CUNEF Universidad, Madrid 28040, Spain} 
\email{pablo.berna@cunef.edu}

\author{Hung Viet Chu}
\address{H. V. Chu, Department of Mathematics, Texas A\&M University, TX 77843, USA}
\email{hungchu1@tamu.edu}

\begin{abstract}
It is known that a basis is almost greedy if and only if the thresholding greedy algorithm gives essentially the smallest error term compared to errors from projections onto intervals or in other words, consecutive terms of $\mathbb{N}$. In this paper, we fix a sequence $(a_n)_{n=1}^\infty$ and compare the TGA against projections onto consecutive terms of the sequence and its shifts. We call the corresponding greedy-type condition the $\mathcal{F}_{(a_n)}$-almost greedy property. Our first result shows that the $\mathcal{F}_{(a_n)}$-almost greedy property is equivalent to the classical almost greedy property if and only if $(a_n)_{n=1}^\infty$ is bounded. Then we establish an analog of the result for the strong partially greedy property. Finally, we show that under a certain projection rule and conditions on the sequence $(a_n)_{n=1}^\infty$, we obtain a greedy-type condition that lies strictly between the almost greedy and strong partially greedy properties.
\end{abstract}

\subjclass[2020]{41A65; 46B15}

\keywords{Thresholding Greedy Algorithm; consecutive almost greedy; bases}

\thanks{Miguel Berasategui was partially supported by the grants CONICETPIP 11220200101609CO y ANPCyT PICT 2018-04104 (Consejo Nacional de Investigaciones Cient\'{i}ficas y T\'{e}cnicas y Agencia Nacional de Promoci\'{o}n de la Investigaci\'{o}n, el Desarrollo Tecnol\'{o}gico y la Innovaci\'{o}n, Argentina). P. M. Bern\'{a} was partially supported by the Grant PID2022-1422ONB-I00 (Agencia Estatal de Investigaci\'{o}n, Spain)}

\maketitle

\section{Introduction}\label{intro}
In the setting of abstract normed vector spaces, the \textit{thresholding greedy algorithm} (TGA) arises as one of the most natural approximation methods. The TGA was first formalized by Konyagin and Temlyakov \cite{KT1999} as follows: let $\mathcal{B} = (e_n)_{n=1}^\infty$ be a basis of the space $(X, \|\cdot\|)$ and let $x\in X$; an $m$-term approximation $G_m(x)$ of $x$ is the projection of $x$ onto $m$ basis vectors whose coefficients, with respect to $x$, are the largest. To evaluate the efficiency of this algorithm, Konyagin and Temlyakov \cite{KT1999} compared the error $\|x-G_m(x)\|$ against the smallest possible error, denoted by $\sigma_m(x)$, from approximation by any $m$-term linear combination of basis vectors. If there is a $C > 0$ such that 
$$\|x-G_m(x)\|\ \leqslant\ C\sigma_m(x), \mbox{ for all }x\in X\mbox{ and }m\in \mathbb{N},$$ then the basis $\mathcal{B}$ is said to be \textit{greedy}. Later, Dilworth et al. \cite{DKKT2003} compared $\|x-G_m(x)\|$ against the smallest error, denoted by $\widetilde{\sigma}_m(x)$, from projections of $x$ onto any $m$ basis vectors. They called a basis \textit{almost greedy} if there is a $C>0$ satisfying 
$$\|x-G_m(x)\|\ \leqslant\ C\widetilde{\sigma}_m(x),\mbox{ for all }x\in X\mbox{ and }m\in \mathbb{N}.$$ Clearly, $\sigma_m(x)\leqslant \widetilde{\sigma}_m(x)$ because projections form a proper subset of linear combinations; hence, a greedy basis is almost greedy. However, an almost greedy basis is not necessarily greedy (see \cite{KT1999} or \cite[Example 10.2.9]{AK}).

The literature has shown many pleasant and often unexpected equivalences among different greedy-type properties. For example, 
Dilworth et al. showed that enlarging greedy sums linearly in the definition of greedy bases brings us to the realm of almost greedy bases (see \cite[Theorem 3.3]{DKKT2003}). Extension of this result was recently obtained in \cite{Ch}. Moreover, Dilworth et al. \cite{DKK2003} introduced the so-called \textit{semi-greedy} bases and proved the equivalence between semi-greedy and almost greedy Schauder bases under the assumption of finite cotype, a condition later removed by the second named author \cite{B1}. Then the first named author and Lassalle extended the equivalence to Markushevich bases \cite{BL2021}. Last but not least, there came the two surprising results by Albiac and Ansorena: first, a basis is \textit{$1$-quasi-greedy} if and only if it is $1$-suppression unconditional \cite{AA2016}, where $1$-quasi-greediness is the property that $\|G_m(x)\|\leqslant \|x\|$ for all $x\in X$ and $m\in \mathbb{N}$; second, a basis is almost greedy with constant $1$ if and only if it has the so-called \textit{Property (A)} \cite{AA2017}. The above-mentioned equivalence between $1$-quasi-greediness and $1$-suppression unconditionality is particularly interesting as it connects a nonlinear property with a linear property. In the same vein, recent work of the three authors Albiac, Anserona, and the first named author \cite{AAB2023a, AAB2023b}  established equivalences among weaker versions of the quasi-greedy property and weaker forms of unconditionality. 

Previously, the authors of the present paper gave an order-dependent definition of almost greedy bases, which turned out to be equivalent to the classical order-free definition \cite{BBCa, BBCb}. In particular, we call a basis \textit{consecutive almost greedy} if the error $\|x-G_m(x)\|$ is essentially bounded by the error $\|x-P_I(x)\|$, where $P_I(x)$ is the projection of $x$ onto an arbitrary interval $I\subset \mathbb{N}$ with $|I| = m$. On the surface, the definition of consecutive almost greedy bases is a weaker notion than the classical definition of almost greedy bases. However, \cite[Theorem 2.8]{BBCa} states that a basis is almost greedy if and only if it is consecutive almost greedy. This is unexpected, because unlike the almost greedy property, the consecutive almost greedy property, which involves projections onto intervals, depends on how we order basis vectors. 
The case for \textit{consecutive greedy basis} is later studied in \cite{BBCb}, where we showed that a basis is consecutive greedy if and only if it is almost greedy and Schauder. Hence, for Schauder bases, consecutive greedy and almost greedy bases coincide. In general, consecutive greedy bases are not necessarily greedy (see \cite[Remark 3.9]{BBCb}). This is where the consecutive greedy case is different from the consecutive almost greedy case.

Our main goal in this paper is to study the interval version with respect to an arbitrary sequence. In this more general setting, we answer the question for which sequences, the consecutive greedy-type notion is equivalent to the classical greedy-type notion. This extends and offers new insight about the mechanics behind \cite[Theorem 2.8]{BBCa} and \cite[Theorem 3.7]{BBCb}. 

\section{Preliminaries and main results}
We will work in the setting of $p$-Banach spaces. Let $X$ be a separable, infinite dimensional $p$-Banach space $(0 < p \leqslant 1)$ over the field $\mathbb{F} = \{\mathbb{R}, \mathbb{C}\}$. Let $X^*$ be the dual of $X$. 

\begin{defi}\normalfont\label{definitionbasis}
A basis $\mathcal{B}$ of $X$ is a sequence $(e_n)_{n=1}^\infty$ in $X$, which has three properties: first, the span of $(e_n)_n$ is dense in $X$; second, there is a unique sequence of functionals $(e^*_n)_n\subset X^*$ such that $e^*_m(e_n) = \delta_{m,n}$ for all $m, n\in \mathbb{N}$; third, the system $(e_n, e_n^*)_n$ is semi-normalized, i.e., 
$$0\ <\ c_1\ :=\ \inf_n \{\|e_n\|, \|e_n^*\|\}\ \leqslant\ \sup_n \{\|e_n\|, \|e_n^*\|\}\ =: \ c_2 \ <\ \infty.$$

If a basis $\mathcal{B}$ has $\overline{\spann(e_n^*)_n}^{w*} = X^*$,  then $\mathcal{B}$ is called a Markushevich basis. If, additionally, there is $C > 0$ such that 
$$\left\|\sum_{n=1}^m e_n^*(x)e_n\right\|\ \leqslant\ C\|x\|, \mbox{ for all }x\in X \mbox{ and }m\in \mathbb{N},$$
then $\mathcal{B}$ is called a Schauder basis. 
\end{defi}

We discuss several important operators. For $m\geqslant 0$, the $m$\textsuperscript{th} partial sum operator $S_m : X\rightarrow X$ is $\sum_{n=1}^m e_n^*(x)e_n$. Hence, Schauder bases have uniformly bounded partial sum operators. Partial sum operators belong to the more general collection of projections. Let $\mathbb{N}^{<\infty}$ consist of all finite subsets of $\mathbb{N}$. For $A\in \mathbb{N}^{<\infty}$, the projection $P_A: X\rightarrow X$ is defined as $P_A(x) = \sum_{n\in A}e_n^*(x)e_n$. Next, we define, for each $m\in \mathbb{N}$, the greedy sum operator, $G_m: X\rightarrow X$:  let $\Lambda_m(x)$ be a set satisfying 
\begin{equation}\label{e101}|\Lambda_m(x)| \ =\ m\mbox{ and }\min_{n\in \Lambda_m(x)}|e_n^*(x)|\ \geqslant\ \max_{n\notin \Lambda_m(x)}|e_n^*(x)|,\end{equation}
then $G_m(x) = \sum_{n\in \Lambda_m(x)}e_n^*(x)e_n$. Unlike $P_A$, $G_m$ is not linear, and $G_m(x)$ depends on the choice of $\Lambda_m(x)$ and thus, may not be unique. However, uniqueness is not a concern thanks to the standard perturbation argument. Using the above notation, we can write $\sigma_m(x)$ and $\widetilde{\sigma}_m(x)$ mentioned in Section \ref{intro} as 
\begin{align}
\sigma_m(x) &\ :=\ \inf\left\{\left\|x-\sum_{n\in A}a_ne_n\right\|\,:\, |A| = m, (a_n)_{n\in A}\subset \mathbb{F}\right\}\mbox{ and }\nonumber\\
\label{e50}\widetilde{\sigma}_m(x)&\ :=\ \inf\left\{\left\|x-P_A(x)\right\|\,:\, |A| = m\right\},
\end{align}
respectively. As observed in \cite[Lemma 3.4]{AA2017}, the requirement ``$|A| = m$" in \eqref{e50} can be replaced by ``$|A| \leqslant m$" without changing the constant of the almost greedy property. 

A basis is said to be $C$-suppression quasi-greedy for some $C > 0$ if 
$$\|x-G_m(x)\|\ \leqslant\ C\|x\|, \mbox{ for all }x\in X, m\in \mathbb{N}, \mbox{ and }G_m(x).$$
Besides greedy-type bases mentioned in Section \ref{intro}, another well-known type is the so-called \textit{strong partially greedy} bases (see \cite{BBL2021, DKKT2003}). A basis is said to be strong partially greedy if there is $C > 0$ such that 
$$\|x-G_m(x)\|\ \leqslant\ C\inf\left\{\|x-S_n(x)\|\,:\, 0\leqslant n\leqslant m\right\}.$$

Let $\mathcal{F}$ be a nonempty, infinite collection of  finite subsets of $\mathbb{N}$ and $$\mathcal{PF}\ :=\ \{A\subset \NN: \exists B\in \mathcal{F}\, :\, A\subset B\}.$$ Unlike \cite{BC}, our main results do not assume that $\mathcal{F}$ is hereditary. With an abuse of notation, let $\min \mathcal{F} := \min_{F\in\mathcal{F}}|F|$. We define $\mathcal{F}$-almost greedy bases, an extension of the classical almost greedy bases. 

\begin{defi}(\cite[Definition 1.6]{BC})\normalfont\label{ }
A basis is said to be $\mathcal{F}$-almost greedy if there exists $\Delta \geqslant 1$ such that 
\begin{equation}\label{e4}\|x-G_m(x)\| \ \leqslant\ \Delta\inf_{F\in \mathcal{F}, |F|\leqslant m}\|x-P_{F}(x)\|, \mbox{ for all } x\in X, m\in \mathbb{N}_{\geqslant \min \mathcal{F}}, \mbox{ and } G_m(x).\end{equation}
\end{defi}

Clearly, an almost greedy basis is $\mathcal{F}$-almost greedy. Our first result gives sufficient conditions on $\mathcal{F}$ such that a basis is $\mathcal{F}$-almost greedy if and only if it is almost greedy. For the definition of the $N$-covering and $N$-sliding properties, see Definitions \ref{definitionNcovering} and \ref{definition subsliding}. 

\begin{thm}\label{m1}Fix an $N$-covering and $M$-sliding family $\mathcal{F}$. Then a basis $\mathcal{B}$ is $\mathcal{F}$-almost greedy if and only if it is almost greedy. 
\end{thm}

When $\mathcal{F}$ is the collection of intervals in $\mathbb{N}$, the $\mathcal{F}$-almost greedy property is what we call the consecutive almost greedy property (\cite[Definition 2.7 and Theorem 2.8]{BBCa}). Our next result is motivated by the following theorem from \cite{BBCa}
\begin{thm}(\cite[Theorem 2.8]{BBCa})
A basis is consecutive almost greedy if and only if it is almost greedy. 
\end{thm}
The theorem states that we can preserve the almost greedy property by choosing sets in $\mathcal{F}$ to be intervals or in other words, to have no gaps. A natural question is whether introducing gaps still preserves the property. If so, how large can the gaps be?
We use Theorem \ref{m1} to answer this question by studying $\mathcal{F}$-almost greedy bases when $\mathcal{F}$ is defined using a given gap sequence. We show that in this case, $\mathcal{F}$-almost greediness is the same as almost greediness if and only if the gaps are bounded. Specifically, let $(a_n)_{n=1}^\infty$ be a sequence of positive integers. Define 
$$\mathcal{F}_{(a_n)}\ :=\ \{\emptyset\}\cup\{k+\{a_1, a_1+a_2, a_1+a_2+a_3, \ldots, a_1+a_2+\cdots+a_\ell\}\,:\, k\geqslant 0, \ell\geqslant 1\},$$
where for an integer $k$ and a set $A = \{m_1, m_2, \ldots, m_\ell\}$, 
$$k+A\ :=\ \{k+m_1, k+m_2, \ldots, k+m_\ell\}.$$
If $a_n = 1$ for all $n$, $\mathcal{F}_{(a_n)}$ is the collection of all finite intervals of $\mathbb{N}$; more generally, when $a_n \equiv d$, $\mathcal{F}_{(a_n)}$ is the collection of all finite arithmetic progressions of difference $d$.  

\begin{thm}\label{m2}
Let $(a_n)_{n=1}^\infty\subset \mathbb{N}$. The following are equivalent
\begin{enumerate}
\item The sequence $(a_n)_{n=1}^\infty$ is bounded.
\item A basis is $\mathcal{F}_{(a_n)}$-almost greedy if and only if it is almost greedy. 
\end{enumerate}
\end{thm}

Furthermore, we obtain the analog of Theorem \ref{m2} for strong partially greedy bases.

\begin{thm}\label{m3}
Let $(a_n)_{n=1}^\infty\subset \mathbb{N}$. The following are equivalent
\begin{enumerate}
\item The sequence $(a_n)_{n=1}^\infty$ is bounded.
\item A basis is $\mathcal{F}_{(a_n)}$-strong partially greedy if and only if it is strong partially greedy. 
\end{enumerate}
\end{thm}

Our final results are concerned with what we call the $\mathcal{F}_{(a_n)}$-minimum partially greedy property inspired by \cite{Ch2}. Depending on the sequence $(a_n)$, we show that the $\mathcal{F}_{(a_n)}$-minimum partially greedy property may be equivalent to strong partial greediness, to almost greediness, or to properties that lie strictly between the two. The following is proved in Section \ref{mpg}.

\begin{thm}\label{corollarystrict}Let $(a_n)_{n\in\NN}$ be a sequence with the following properties
\begin{itemize}
\item $(a_n)_{n\in \NN}$ is bounded. 
\item There is $\alpha >0$ such that for all $m\in \mathbb{N}$, 
$$
\left| \left \{ 1\leqslant n\leqslant m\, :\, a_n\geqslant 2\right\} \right| \ \leqslant\ \alpha m^{1/4}.
$$
\item The set $\left| \left \{ 1\leqslant n\leqslant m: a_n\geqslant 2\right\}\right|$ is infinite. 
\end{itemize}
Then the $\cF_{(a_n)}$-minimum partially greedy property lies strictly between strong partial greediness and almost greediness. 
\end{thm}

\section{Framework for an infinite family $\mathcal{F}$ of finite subsets}

The main purpose of this section is to find sufficient conditions on $\mathcal{F}$ such that a basis is $\mathcal{F}$-almost greedy if and only if it is almost greedy. Before diving into proving Theorem \ref{m1}, we go over several useful existing results. In what follows, we employ the following notation
\begin{enumerate}
\item A sign $\varepsilon = (\varepsilon_n)$ is a sequence of scalars of modulus $1$. 
\item For $A\in \mathbb{N}^{<\infty}$ and a sign $\varepsilon$, let $1_A = \sum_{n\in A}e_n$ and $1_{\varepsilon, A} = \sum_{n\in A}\varepsilon_n e_n$.
\item For two subsets $A, B\subset \mathbb{N}$, $A < B$ means that $a < b$ for all $a\in A$ and $b\in B$. Corresponding meanings hold for inequalities $\leqslant, >, \geqslant$.
\end{enumerate}

\begin{defi}\normalfont(\cite{BC, KT1999}) \label{Fdemocracy}
A basis $\mathcal{B}$ is $\mathcal{F}$-(disjoint, respectively) democratic if there is $C > 0$ such that $\|1_A\|\leqslant C\|1_B\|$ for any (disjoint, respectively) $A, B\in \mathbb{N}^{<\infty}$ with $|A|\leqslant |B|$ and $A\in \mathcal{F}$. 

If we can replace $1_A$ and $1_B$ with $1_{\varepsilon,A}$ and $1_{\delta, B}$ for any signs $\varepsilon, \delta$, we say that $\BB$ is $\mathcal{F}$-(disjoint) superdemocratic. 

If $\mathcal{B}$ is $\mathcal{F}$-(super)democratic and $\mathcal{F} = \mathcal{P}(\mathbb{N})$, we say that $\mathcal{B}$ is (super)democratic (as in the classical definition by Konyagin and Temlyakov \cite{KT1999}).
\end{defi}

The following is a neat characterization of almost greedy bases.
\begin{thm}(\cite[Theorem 3.3]{DKKT2003} and \cite[Theorem 6.3]{AABW})\label{dkkt}\label{bc}
A basis is almost greedy if and only if it is quasi-greedy and democratic. 
\end{thm}
Beanland and the third named author recently extended the above theorem to any hereditary family $\mathcal{F}$ \cite{BC}.  Recall that a family $\mathcal{F}$ of sets is said to be hereditary if whenever $A\in \mathcal{F}$ and $B\subset A$, we have $B\in \mathcal{F}$. While the original proof was for Banach spaces, the theorem holds for $p$-Banach spaces with obvious modifications. 
\begin{thm}(\cite[Theorem 1.7]{BC})\label{bcal}
Let $\mathcal{F}$ be a hereditary family. A basis is $\mathcal{F}$-almost greedy if and only if it is quasi-greedy and $\mathcal{F}$-disjoint democratic. 
\end{thm}

We are ready to introduce the two properties, $N$-covering and $N$-slicing, which guarantee the desired equivalence between $\mathcal{F}$-almost greediness and almost greediness. 

\begin{defi}\label{definitionNcovering}\normalfont For $N\in\mathbb{N}$, a family $\mathcal{F}$ is said to be $N$-covering if for all $B\in \mathbb{N}^{<\infty}$, there exist $(F_i)_{i=1}^N \subset \mathcal{F}$ such that
$$\left|B\backslash (\cup_{i=1}^N F_i)\right| \ \leqslant\ N.$$
\end{defi}

\begin{defi}\label{definition subsliding}\normalfont
For $N\in \mathbb{N}_{\geqslant 0}$, a family $\mathcal{F}$ is said to be $N$-sliding if for all $M\in \mathbb{N}$, there exists $F\in \mathcal{F}$ such that 
\begin{align}
&\left|F\right|\ \geqslant\ M, \mbox{ and }\\
&\left|F\cap \{1,2, \ldots, M\}\right|\ \leqslant\ N.
\end{align}
\end{defi}

The inclusion of the empty set in $\mathcal{F}$ preserves the $\mathcal{F}$-almost greedy property (up to the constant), the $N$-covering property, and the $N$-sliding property. The last two are obvious; we show the following.

\begin{prop}
A basis $\mathcal{B}$ is $\mathcal{F}$-almost greedy if and only if it is $\mathcal{F}\cup\{\emptyset\}$-almost greedy.
\end{prop}

\begin{proof}
The backward implication is obvious. For the forward implication, assume that $\mathcal{B}$ is $\Delta$-$\mathcal{F}$-almost greedy. Let $x\in X$ with a greedy sum $G_m(x)$ for some $m\geqslant 0$. Let $F\in \mathcal{F}\cup\{\emptyset\}$ with $|F|\leqslant m$. If $F \neq \emptyset$, then $F\in \mathcal{F}$, and $\mathcal{F}$-almost greediness gives
$$\|x-G_m(x)\|\ \leqslant\ \Delta\|x-P_F(x)\|.$$
If $F = \emptyset$, we shall show that 
$$\|x-G_m(x)\|\ \leqslant\ \Delta(1+c_2^{2p}(\min \mathcal{F})^p)^{1/p}\|x\|.$$
Indeed, if $m\leqslant \min \mathcal{F}$, then 
\begin{align*}\|x - G_m(x)\|^p\ \leqslant\ \|x\|^p + \|G_m(x)\|^p&\ \leqslant\ \|x\|^p + c_2^{2p}(\min \mathcal{F})^p\|x\|^p\\
&\ =\ (1+c_2^{2p}(\min \mathcal{F})^p)\|x\|^p.\end{align*}
If $m > \min\mathcal{F}$, let $G \in \mathcal{F}$ such that $|G| = \min \mathcal{F}$. Then due to $\mathcal{F}$-almost greediness and the above case, 
$$\|x-G_m(x)\|\ \leqslant\ \Delta\|x-P_G(x)\|\ \leqslant\ (1+c_2^{2p}(\min \mathcal{F})^p)^{1/p}\Delta\|x\|.$$
Hence, our basis is $\mathcal{F}\cup\{\emptyset\}$-almost greedy with constant $(1+c_2^{2p}(\min\mathcal{F})^p)^{1/p}\Delta$. 
\end{proof}

Therefore, we can and shall assume that $\emptyset\in \mathcal{F}$.

\begin{lem}\label{l0}
Let $\mathcal{B}$ be $\Delta$-$\mathcal{F}$-almost greedy. Then $\mathcal{B}$ is $\Delta$-suppression quasi-greedy.
\end{lem}
\begin{proof}
Since $\emptyset\in \mathcal{F}$, the result follows by setting $F = \emptyset$ in \eqref{e4}.
\end{proof}

For any family $\mathcal{F}$, $\mathcal{PF}$ is hereditary, and, if $\mathcal{F}$ is hereditary, then $\mathcal{PF}=\mathcal{F}$.  Furthermore, if $\mathcal{F}$ is hereditary and contains sets of arbitrarily large cardinality, it is $0$-sliding. On the other hand, it is easy to check that a $0$-sliding family $\mathcal{F}$ need not be hereditary: just take a family of disjoint intervals $(I_n)_{n\in \NN}$ of unbounded cardinality with $I_n<I_{n+1}$ for all $n\in \NN$.  Given a sliding family $\mathcal{F}$, $\mathcal{F}$-almost greediness is equivalent  to $\mathcal{P}\mathcal{F}$-almost greediness up to a constant. More precisely, we have the following result. 

\begin{lem}\label{lemmasubsets} Fix an $N$-sliding family $\mathcal{F}$. If a basis $\BB$ is $\Delta$-$\mathcal{F}$-almost greedy, then $\BB$ is $\max\{2^{1/p}\Delta^3, (2N)^{1/p}c_2^2\}$-$\mathcal{P}\mathcal{F}$-superdemocratic. Hence,  it is $\mathcal{P}\mathcal{F}$-almost greedy. 
\end{lem}
\begin{proof}
To prove the $\mathcal{P}\mathcal{F}$-superdemocracy property, fix $A\in \mathcal{PF}$ and $B\in \NN^{<\infty}$ with $|B|\geqslant |A|$, and signs $\varepsilon, \delta$. First, suppose $|A|> 2N$. Pick $A_1\in \mathcal{F}$ so that $A\subset A_1$ and choose $M>N+\max(B\cup A_1)$. By hypothesis, there is $F\in \mathcal{F}$ such that 
$$
\left| F\cap \{1,\dots, M\}\right|\ \leqslant\ N\mbox{ and }|F|\ \geqslant\ M. 
$$
As 
$$|F\cap\{1, \ldots, M\}|\ \leqslant\ N\mbox{ and } |F|-N\ \geqslant\ M- N\ >\ \max(B\cup A_1)\ >\ |A|,$$
there is $F_1\subset F$ such that $|F_1|=|A|$ and $F_1>B\cup A_1$. By $\mathcal{F}$-almost greediness, we have
\begin{align*}
\|1_{\varepsilon, A}\|&\ =\ \|1_{\varepsilon, A}+1_{A_1\setminus A}+1_{F_1}-(1_{A_1\setminus A}+1_{F_1})\|\\
&\ \leqslant \ \Delta\|1_{\varepsilon, A}+1_{A_1\setminus A}+1_{F_1}-P_{A_1}(1_{\varepsilon, A}+1_{A_1\setminus A}+1_{F_1})\|\ =\ \Delta\|1_{F_1}\|. 
\end{align*}
Let $B_1:=B\setminus F$. Since 
$$|B_1|\ \geqslant\ |B|-N \ >\ \frac{|B|}{2} \ \geqslant\ \frac{|A|}{2}\ =\ \frac{|F_1|}{2},$$
there is a partition $\{F_2,F_3\}$ of $F_1$ with $|F_j|\leqslant |B_1|$, $j=2, 3$. By Lemma \ref{l0}, for $j=2, 3$, 
\begin{align*}
\|1_{F_j}\|&\ =\ \|1_{F_j}+1_{F\setminus F_j}+1_{\delta, B_1}-(1_{F\setminus F_j}+1_{\delta, B_1})\|\\
&\ \leqslant\ \Delta\|1_{F_j}+1_{F\setminus F_j}+1_{\delta, B_1}-P_F(1_{F_j}+1_{F\setminus F_j}+1_{\delta, B_1})\|\\
&\ =\ \Delta\|1_{\delta, B_1}\|\ \leqslant\ \Delta^2\|1_{\delta,B}\|. 
\end{align*}
Combining the above inequalities, we obtain 
$$
\|1_{\varepsilon,A}\|^p\ \leqslant\ \Delta^p\|1_{F_1}\|^p\ \leqslant\ \Delta^p (\|1_{F_2}\|^p+\|1_{F_3}\|^p)\ \leqslant\ 2\Delta^{3p}\|1_{\delta,B}\|^p. 
$$
Hence, $\|1_{\varepsilon, A}\|\leqslant\ 2^{1/p}\Delta^3\|1_{\delta, B}\|$.
On the other hand, if $|A|\leqslant 2N$, then we have the trivial bounds
$$
\|1_{\varepsilon,A}\|^p\ \leqslant\  |A|c_2^p\ \leqslant\ 2Nc_2^p\mbox{ and }1\ \leqslant\ c_2\|1_{\delta, B}\|,
$$
which give
$$
\|1_{\varepsilon, A}\|\ \leqslant\ (2N)^{1/p}c_2^2\|1_{\delta,B}\|. 
$$
Hence, $\BB$ is $\mathcal{P}\mathcal{F}$-superdemocratic, with the constant as in the statement. Given that $\mathcal{P}\mathcal{F}$ is hereditary, by Theorem \ref{bc} and Lemma \ref{l0}, $\BB$ is $\mathcal{P}\mathcal{F}$-almost greedy. 
\end{proof}

\begin{rek}\rm \label{remarkspg} Generally, $\mathcal{F}$-almost greediness is not equivalent to $\mathcal{P}\mathcal{F}$-almost greediness.  Indeed, let $\mathcal{F}_0$ be the family consisting of the empty set and all finite initial segments of $\NN$. Then it follows from the definition that a basis $\BB$ is $\mathcal{F}_0$-almost greedy if and only if it is strong partially greedy. It is known that this property does not entail almost greediness (see \cite{DKKT2003}). On the other hand, it is obvious that $\BB$ is $\mathcal{P}\mathcal{F}_0$-almost greedy if and only if it is almost greedy. 
\end{rek}

\begin{proof}[Proof of Theorem \ref{m1}]
Suppose that $\BB$ is $\Delta$-$\mathcal{F}$-almost greedy. By Lemma~\ref{lemmasubsets}, the basis $\BB$ is $\Delta_1$-$\mathcal{P}\mathcal{F}$-superdemocratic, 
with 
$$\Delta_1\ =\ \max\{2^{1/p}\Delta^3, (2N)^{1/p}c_2^2\}.$$ Fix $A, B\in \NN^{<\infty}$ with $|A|\leqslant |B|$ and signs $\varepsilon, \delta$. By hypothesis, there are $(F_i)_{i=1}^{N}\subset \mathcal{F}$ such that 
$$
|A_0\ :=\ A\setminus \left(\cup_{i=1}^{N}F_i\right)|\ \leqslant\ N. 
$$
Choose pairwise disjoint (possibly empty) sets $(A_i)_{i=1}^{N}$ so that 
$$
A_i\subset F_i, \mbox{ for all }1\leqslant i\leqslant N\mbox{ and }A\setminus A_0\ =\ \sqcup_{i=1}^{N}A_i. 
$$
By the triangle inequality, 
$$
\|1_{\varepsilon,A}\|^p\ \leqslant\  \sum_{i=0}^{N}\|1_{\varepsilon, A_i}\|^p\ \leqslant\ Nc_2^{2p}\|1_{\delta,B}\|^p+\sum_{i=1}^{N}\Delta_1^p\|1_{\delta,B}\|^p\ =\ N(c_2^{2p}+\Delta^p_1)\|1_{\delta,B}\|^p. 
$$
This proves that $\BB$ is $N^{1/p}(c_2^{2p}+\Delta_1^p)^{1/p}$-superdemocratic. Given that $\mathcal{B}$ is quasi-greedy by Lemma \ref{l0}, Theorem \ref{dkkt} guarantees that it is almost greedy. 
\end{proof}

\begin{rek}\rm Note that the equivalence in Theorem \ref{m1} does not hold without the sliding condition, as the example in Remark~\ref{remarkspg} shows. 
\end{rek}

To close this section, we extend Theorem \ref{bcal} to families that are not necessarily hereditary. We first need an extension of (super)democracy to our context. 

\begin{defi}\label{Fdemocracy2}\normalfont A  basis $\mathcal{B}$ is $\mathcal{F}$-strong disjoint democratic if there is $C > 0$ such that $\|1_A\|\leqslant C\|1_B\|$ for any $A, B\in \mathbb{N}^{<\infty}$ such that $|A|\leqslant |B|$, and there is $F\in \cF$ such that $A\subset F$ and $B\cap F=\emptyset$. 

If we can replace $1_A$ and $1_B$ with $1_{\varepsilon,A}$ and $1_{\delta, B}$ for any signs $\varepsilon, \delta$, we say that $\BB$ is $\mathcal{F}$-strong disjoint superdemocratic. 
\end{defi}
\begin{rek}\label{remarkFdisjoint}\normalfont If $\cF$ is a hereditary family, a basis is $C$-disjoint (super)democratic if and only if it is $C$-strong disjoint (super)democratic. 
\end{rek}

To extend Theorem \ref{bcal}, we also need some auxiliary results from the literature. For $x\in X$, let $\sgn(e_n^*(x)) = \begin{cases}e_n^*(x)/|e_n^*(x)|&\mbox{ if }e_n^*(x)\neq 0,\\ 1&\mbox{ otherwise}\end{cases}$ and $\varepsilon(x) = (\sgn(e_n^*(x)))$.

\begin{defi}\label{definitionTQG}\normalfont Let $\BB$ be a basis of a $p$-Banach space $X$. We say that $\BB$ is \emph{truncation quasi-greedy} if there is $C>0$ such that 
\begin{equation*}
\min_{n\in \Lambda_m(x)}|e_n^*(x)|\|1_{\varepsilon(x),\Lambda_m(x)}\|\ \leqslant\ C\|x\|
\end{equation*}
for all $x\in X$, $m\in \NN$, and $\Lambda_m(x)$. 
\end{defi}

\begin{lem}(\cite[Theorem 4.13 and Proposition 4.16]{AABW} and \cite[Lemma 2.2]{DKKT2003})\label{lemmaTQG} Let $\BB$ be a basis of a $p$-Banach space $X$. Then 
\begin{itemize}
\item If $\BB$ is quasi-greedy, it is truncation quasi-greedy. 
\item If $\BB$ is truncation quasi-greedy, there is $C>0$ such that
\begin{equation}
\left\|\sum_{n\in A}a_ne_n\right\|\ \leqslant \ C\|x\|,\label{tqg+}
\end{equation}
for all $A\in \NN^{<\infty}$, scalars $(a_n)_{n\in \NN}$, and $x\in X$ such that 
$$
\max_{n\in A}|a_n|\ \leqslant\ \min_{n\in A}|e_n^*(x)|. 
$$
\end{itemize}
\end{lem}

\begin{thm}(Extension of \cite[Theorem 1.7]{BC})\label{bcal2}
Let $\mathcal{F}$ be a family, and $\BB$ a basis of a $p$-Banach space $X$. The following are equivalent
\begin{enumerate}[\rm i)]
\item \label{FAG}$\BB$ is $\cF$-almost greedy. 
\item \label{QGFSDS}$\BB$ is quasi-greedy and  $\mathcal{F}$-strong disjoint superdemocratic. 
\item \label{QGFSDD} $\BB$ is quasi-greedy and  $\mathcal{F}$-strong disjoint democratic. 
\end{enumerate}
\end{thm}
\begin{proof}
Let us show that \ref{FAG} implies \ref{QGFSDS}. Suppose that $\BB$ is $\Delta$-$\cF$-almost greedy. By Lemma~\ref{l0}, $\BB$ is quasi-greedy.  To see that it is $\cF$-strong disjoint superdemocratic, fix $F\in \cF$, $A\subset F$, and $B\in \NN^{<\infty}$ so that  $F\cap B=\emptyset$ and $|A|\leqslant |B|$. Fix signs $\varepsilon$, $\eta$. We have 
\begin{align*}
\|1_{\varepsilon, A}\|&\ =\ \|1_{\varepsilon, A}+1_{F\setminus A}+1_{\eta, B}-(1_{F\setminus A}+1_{\eta, B})\|\\
&\ \leqslant\  \Delta  \|1_{\varepsilon, A}+1_{F\setminus A}+1_{\eta, B}-P_F(1_{\varepsilon, A}+1_{F\setminus A}+1_{\eta, B})\|=\Delta\|1_{\eta,B}\|. 
\end{align*}

That \ref{QGFSDS} implies \ref{QGFSDD} is immediate. 

Finally, to see that \ref{QGFSDD} implies \ref{FAG}, we let $C$ be a constant for which \eqref{tqg+} holds, and suppose that $\mathcal{B}$ is $C_q$-suppression quasi-greedy and $C_d$-$\cF$-strong disjoint democratic. Fix $x\in X$, $m\in \NN$, $\Lambda_m(x)$, and $F\in \cF$ with $|F|\leqslant m$. We have
\begin{align*}
\|x-P_{\Lambda_m(x)}(x)\|^p&\ =\ \|x-P_{F\cap \Lambda_m(x)}(x)- P_{\Lambda_m(x)\setminus F}(x-P_{F\cap \Lambda_m(x)}(x))  \|^p\\
&\ \leqslant\ C_q^p\|x-P_{F\cap \Lambda_m(x)}(x)\|^p\\
&\ \leqslant\ C_q^p\|x-P_F(x)\|^p+C_q^p\|P_{F\setminus \Lambda_m(x)}(x)\|^p\\
& \ \leqslant\ C_q^p\|x-P_F(x)\|^p+C_q^pC^p\max_{n\in F\setminus \Lambda_m(x)}|e_n^*(x)|^p\|1_{F\setminus \Lambda_m(x)}\|^p\\
& \leqslant\  C_q^p\|x-P_F(x)\|^p+C_q^pC^pC_d^p\min_{n\in  \Lambda_m(x)\setminus F}|e_n^*(x)|^p\|1_{ \Lambda_m(x)\setminus F }\|^p\\
& \leqslant\  C_q^p(1+C_d^pC^{2p})\|x-P_F(x)\|^p. 
\end{align*}
Hence, $\mathcal{B}$ is $\cF$-almost greedy with a constant no greater than $C_q(1+C_d^pC^{2p})^{1/p}$. 
\end{proof}

\section{Equivalence requires bounded gaps}

We split Theorem \ref{m2} into two theorems. The proof of the first one is an application of Theorem \ref{m1}, while the second one calls for constructing a basis with certain properties.  

\begin{thm}\label{theorembounded}
If $(a_n)_n$ is bounded, then a basis is $\mathcal{F}_{(a_n)}$-almost greedy if and only if it is almost greedy. 
\end{thm}

\begin{proof}
Let $(a_n)_n$ be bounded. It is easy to see that $\mathcal{F}_{(a_n)}$ is $0$-sliding. By Theorem \ref{m1}, we need only to verify that $\mathcal{F}_{(a_n)}$ is $M$-covering for $M = \max a_n$. Choose $B\in \mathbb{N}^{<\infty}$. Let $B_1 = \{n\in B: n < M\}$ and $B_2 = \{n\in B: n\geqslant M\}$. First, suppose that $|B_2|\geqslant 2$; hence, $\max B_2\geqslant 2$.

Consider $M$ following sets in $\mathcal{F}_{(a_n)}$:
\begin{align*}
S_1 &\ :=\ (\min B_2-a_1) + \{a_1, a_1 + a_2, \ldots, a_1 + \cdots + a_{\max B_2}\},\\
S_2 &\ :=\ (\min B_2 - a_1 + 1) + \{a_1, a_1 + a_2, \ldots, a_1 + \cdots + a_{\max B_2}\},\\
&\ \vdots\ \\
S_{M} &\ :=\ (\min B_2 - a_1 + (M-1)) + \{a_1, a_1 + a_2, \ldots, a_1 + \cdots + a_{\max B_2}\}.
\end{align*}
Pick $k\in B_2$. Suppose, for a contradiction, that $k\notin \cup_{i=1}^{M} S_i$. Then $k\geqslant \min B_2+M$ because $\{\min B_2, \min B_2+1,  \ldots, \min B_2+M-1\}\subset \cup_{i=1}^{M} S_i$. Furthermore, $k\notin \cup_{i=1}^{M} S_i$ implies that for all $1\leqslant i\leqslant M$,
\begin{equation}\label{e60}\ell_i \ :=\ k+a_1-\min B_2 - (i-1)\ \notin\ \{a_1, a_1+a_2, \ldots, a_1+\cdots+a_{\max B_2}\}.\end{equation}
Note that for all $1\leqslant i\leqslant M$, 
\begin{equation*}
    \ell_i\ \geqslant\ (\min B_2+M)+ a_1-\min B_2 - (M-1)\ = \ a_1+1.
\end{equation*}
\begin{align*}
\ell_i\ \leqslant\ k+a_1-\min B_2&\ \leqslant\ \max B_2 + a_1  - \min B_2\\
&\ \leqslant\ \max B_2\ =\ \underbrace{1 + \cdots + 1}_{\max B_2}\\
&\ \leqslant\ a_1 + a_2 + \ldots + a_{\max B_2}.
\end{align*}
Hence, $(\ell_i-a_1)_{i=1}^M$ are $M$ consecutive numbers in $\{1, 2, \ldots, a_2 + \cdots + a_{\max B_2}\}$, and $\{\ell_i-a_1\}_{i=1}^M\cap\{a_2, \ldots, a_2 + \cdots + a_{\max B_2}\} = \emptyset$ by \eqref{e60}. This contradicts that $\max a_n = M$. Hence, $B_2\subset \cup_{i=1}^M S_i$. Therefore, we have
$$|B\backslash \cup_{i=1}^M S_i|\ =\ |B_1| + |B_2\backslash  \cup_{i=1}^M S_i|\ \leqslant\ M-1.$$
This resolves the case when $|B_2|\geqslant 2$. If $|B_2|\leqslant 1$, then $|B|\leqslant M$, and for any $(S'_i)_{i=1}^M\subset \mathcal{F}_{(a_n)}$, $|B\backslash \cup_{i=1}^M S'_i|\leqslant M$. We conclude that $\mathcal{F}_{(a_n)}$ is $M$-covering. This completes our proof. 
\end{proof}

\begin{rek}\rm \label{remarksequencehereditary} Since $\mathcal{F}_{(a_n)}$ is $0$-sliding, Lemma~\ref{lemmasubsets} entails that if $\BB$ is $\mathcal{F}_{(a_n)}$-almost greedy, it is $\mathcal{P}\mathcal{F}_{(a_n)}$-almost greedy. 
\end{rek}

Next, we show that the converse of  Theorem~\ref{theorembounded} also holds. In other words, given an unbounded sequence $(a_n)_{n=1}^{\infty}\subset\NN$, we construct bases that are $\mathcal{F}_{(a_n)}$-almost greedy but not almost greedy. Moreover, we will show that this is possible for both conditional and unconditional Schauder bases. 
In our construction, we will use the following result from the literature. 

\begin{lem} (\cite{KT1999}, \cite[page 35]{T2011}, or \cite[page 266]{T2008}) \label{lemmakt} Let $X$ be the completion of $c_{00}$ with the norm 
$$
\left\| x=(x_n)_{n\in \NN}\right\|\ :=\ \max\left\{ \|x\|_{\ell_2},\sup_{m\in \NN}\left| \sum_{n=1}^m \frac{x_n}{n^{1/2}}\right|\right\}. 
$$
Then the canonical unit vector basis of $X$ is quasi-greedy and superdemocratic, with
$$
\|1_{\varepsilon, A}\|\ \approx\ |A|^{1/2},
$$
for every $A\in  \NN^{<\infty}$ and any sign $\varepsilon$. 
\end{lem}

\begin{thm}\label{propunbounded} Let $(a_n)_{n=1}^{\infty}\subset\NN$ be an unbounded sequence. 
\begin{enumerate}
\item \label{conditional}There exists a conditional, quasi-greedy, Schauder basis $\mathcal{B}$ that is not democratic but is $\mathcal{P}\mathcal{F}_{(a_n)}$-superdemocratic. 
In particular, for any sign $\varepsilon$, 
\begin{equation}\label{democracy}
\|1_{\varepsilon, A}\|\ \approx\ |A|^{1/2}, \forall A\in \mathcal{PF}_{(a_n)}\mbox{ and } \|1_{\varepsilon, A}\|\ \geqslant\ |A|^{1/2}, \forall A\in\mathbb{N}^{<\infty}. 
\end{equation}
Hence, $\BB$ is $\mathcal{F}_{(a_n)}$-almost greedy but not almost greedy. 
\item\label{unconditional}There is an unconditional Schauder basis for which \eqref{democracy} holds, but the basis is not democratic. 
\end{enumerate}
\end{thm}

\begin{proof}
(1) To simplify notation, let $\mathcal{F}:=\mathcal{F}_{(a_n)}$. First, note that if $\BB$ is a conditional quasi-greedy Schauder basis for which \eqref{democracy} holds, by Theorem \ref{bcal}, it is $\mathcal{P}\mathcal{F}$-almost greedy and thus, is $\mathcal{F}$-almost greedy. To construct a basis with the desired properties, for each integer $n\in \mathbb{N}$, set
$$
b_n\ :=\ \sum_{k=1}^{n} a_k.
$$
Then 
$$
\mathcal{F}\ =\ \{\emptyset\}\cup \left\{ j+\left\{ b_n\right\}_{1\leqslant n\leqslant \ell}\, :\,\ell\geqslant 1, j\geqslant 0 \right\}. 
$$
Since $(a_n)_n$ is unbounded, we can choose a strictly increasing sequence $(n_k)_{k\geqslant 0}$ recursively so that $n_0 = 1$ and for $k\geqslant 1$,
\begin{equation*}
a_{n_{k}+1}\ >\ 2 k^6+b_{n_{k-1}+1}.
\end{equation*}
Therefore, for all $k\in \mathbb{N}$,
\begin{equation}\label{biggaps}
b_{n_{k}+1}-b_{n_k}\ > \ 2k^6+b_{n_{k-1}+1}.
\end{equation}

Now pick intervals $(A_k)_{k\in \NN}$ of natural numbers so that, for all $k\in \NN$, 
\begin{align}\label{separatedsets1}|A_k|&\ =\ k,\\
\label{separatedsets2}
b_{n_{k}}+b_{{n_{k-1}}+1}+k&\ <\ A_k\ <\ b_{{n_k}+1}. 
\end{align}
For example, we can take
\begin{align*}
A_k&\ =\ \{b_{n_k}+b_{n_{k-1}+1}+k+1, \ldots, b_{n_k}+b_{n_{k-1}+1}+2k\}.
\end{align*}
\begin{claim}\label{claim3}For each $A\in \mathcal{P}\mathcal{F}$, we have 
$$N(A)\ :=\ \left|\{k\in \NN\,:\, A_k\cap A\not =\emptyset\}\right| \ \leqslant\ 1.$$ 
\end{claim}
\begin{proof}
It suffices to prove the claim for $A\in \mathcal{F}$. Suppose $N(A)\geqslant 2$. Pick $j\geqslant 0, \ell\geqslant 1$ so that $A= j+\left\{ b_n\right\}_{1\leqslant n\leqslant \ell}$. Let 
$$
s\ :=\ \max\{k\in \NN\,:\, A_{k+1}\cap A\not=\emptyset\},
$$
and choose $1\leqslant n(A)\leqslant \ell$ so that $j+b_{n(A)}\in A_{s+1}$. It follows from \eqref{separatedsets2} that 
$$b_{n(A)}\ \leqslant\ j+b_{n(A)} \ <\ b_{n_{s+1}+1};$$
hence, $n(A)\leqslant n_{s+1}$. By \eqref{separatedsets2},
$$j+b_{n(A)}\ >\ b_{n_{s+1}}+b_{n_s+1}+s+1,$$
which, combined with $n(A)\leqslant n_{s+1}$, gives
$$j\ >\ b_{n_s+1}+s+1.$$
In particular, $A>b_{n_{s}+1}$. It follows again from \eqref{separatedsets2} that
$$
A\cap A_k\ =\ \emptyset,\mbox{ for all }1\leqslant k\leqslant s, 
$$
a contradiction that proves our claim. 
\end{proof}
Now we construct our basis, borrowing ideas from an example in \cite{KT1999}. 
Pick $1<p<2$ so that if $1/p+1/q=1$, then $1/p-1/q\leqslant 1/2$.
For each $k\in \NN$, write $A_k=\{c_{k,1},\dots,c_{k,k}\}$, and define on $c_{00}$ the seminorm 
$$
\left\| x=(x_n)_{n\in \NN}\right\|_{\circ}\ :=\ \sum_{k\in \mathbb{N}}\sum_{j=1}^{k}\left| x_{c_{k,j}}\left(\frac{k(k-1)}{2}+j\right)^{-1/q} \right|. 
$$
Now define the norm
$$
\left\| x=(x_n)_{n\in \NN}\right\|_{\triangleleft}\ :=\ \sup_{n\in \NN}\left| \sum_{k=1}^{n}\frac{x_k}{k^{1/2}}\right|,
$$
and let $X$ be the completion of $c_{00}$ with respect to the norm
$$
\|x\|\ :=\ \max\{\|x\|_{\circ}, \|x\|_{\triangleleft}, \|x\|_{\ell_2}\}.
$$
Let $\BB$ be $X$'s canonical vector basis. Then $\BB$ is a monotone normalized Schauder basis. Also,  $\BB$ is unconditional with respect to $\|\cdot\|_{\circ}$, and it is known that it is quasi-greedy with respect to $\max\{\|\cdot\|_{\triangleleft}, \|\cdot\|_{\ell_{2}}\}$ (by Lemma~\ref{lemmakt}). Hence, $\BB$ is quasi-greedy with respect to $\|\cdot\|$. To prove \eqref{democracy}, first notice that  for any $A\in \NN^{<\infty}$ and any sign $\varepsilon$,
$$
\|1_{\varepsilon, A}\|\ \geqslant\  \|1_{\varepsilon, A}\|_{\ell_2}\ =\ |A|^{1/2},
$$
whereas 
$$
\|1_{\varepsilon, A}\|_{\triangleleft}\ \leqslant\ \sum_{k=1}^{|A|}\frac{1}{k^{1/2}}\ \leqslant\ 2|A|^{1/2}. 
$$
Now suppose $A\in \mathcal{P}\mathcal{F}$, $A\not=\emptyset$, and $\varepsilon$ is any sign. If $N(A)=0$, then by the above computations, $\|1_{\varepsilon, A}\|\leqslant 2|A|^{1/2}$. Otherwise, by Claim~\ref{claim3}, we have $N(A)=1$. Let $k$ be the unique natural number for which $A_k\cap A\neq\emptyset$. We proceed by case analysis. 

If either $k\leqslant 4$ or $|A|\leqslant 4$,  then 
$$
\|1_{\varepsilon, A}\|_{\circ}\ \leqslant\ |A_k\cap A|\ \leqslant\ 4\ \leqslant\ 4|A|^{1/2}.
$$

If $5\leqslant k\leqslant |A|$, then 
\begin{align*}
\|1_{\varepsilon, A}\|_{\circ}&\ \leqslant\ \sum_{j=1}^{k}\left(\frac{k(k-1)}{2}+j\right)^{-1/q}\ \leqslant\ \left(\frac{k(k-1)}{2}\right)^{-1/q}k\ \leqslant\ 2 k^{1/p}(k-1)^{-1/q}\\
&\ \leqslant\ 2\left(\frac{k}{k-1}\right)^{1/q}k^{1/p-1/q}\ \leqslant\ 4k^{1/2}\ \leqslant\ 4|A|^{1/2}.
\end{align*}

If $5\leqslant |A|\leqslant k$, then 
\begin{align*}
\|1_{\varepsilon, A}\|_{\circ}&\ \leqslant\ \sum_{j=1}^{|A|}\left(\frac{k(k-1)}{2}+j\right)^{-1/q}\ \leqslant\ \left(\frac{|A|(|A|-1)}{2}\right)^{-1/q}|A|\\
&\ \leqslant\ 2 |A|^{1/p}(|A|-1)^{-1/q}\\
&\ \leqslant\ 2\left(\frac{|A|}{|A|-1}\right)^{1/q}|A|^{1/p-1/q}\ \leqslant\ 4|A|^{1/2}. 
\end{align*}
As we have considered all cases, it follows that $\|1_{\varepsilon, A}\|\leqslant 4|A|^{1/2}$, and the proof of \eqref{democracy} is complete. 

To see that $\BB$ is not democratic, for each $k\in \mathbb{N}$, set $D_k:=\cup_{j=1}^{k}A_j$ so that $|D_k|= k(k+1)/2$. Choose an arbitrary $m>4$ and let
$$
s \ :=\ \max\left\{ k\in \mathbb{N}\,:\, k(k+1)/2\leqslant m\right\} \ \geqslant\ 2. 
$$
Then
$$
|D_s|\ \leqslant\ m\ <\ \frac{(s+1)(s+2)}{2} \ \leqslant\ \frac{3s(s+1)}{2}\ =\ 3|D_s|,
$$
and 
\begin{align*}
\|1_{D_{s}}\|_{\circ}&\ = \ \sum_{k=1}^{s}\sum_{j=1}^{k}\left(\frac{k(k-1)}{2}+j\right)^{-1/q}\ =\ \sum_{j=1}^{|D_{s}|}j^{-1/q}\\
&\ \geqslant\  p\left(\left(|D_{s}|+1\right)^{1/p}-1\right)\ \geqslant\ p|D_s|^{1/p}-p\\
&\ =\ \frac{1}{2}|D_s|^{1/p}+\left(p-\frac{1}{2}\right)\left(|D_s|^{1/p}-\frac{2p}{2p-1}\right)\\
&\ \geqslant\ \frac{1}{2} |D_s|^{1/p} \ \geqslant\ \frac{1}{6}m^{1/p}.
\end{align*}

Now pick $B\in \NN^{<\infty}$ so that $|B_m:=D_{s}\sqcup B|=m$ and note that $\|1_{B_m}\|_{\circ}\geqslant \|1_{D_{s}}\|_{\circ}$. Also pick $E_m\in\mathcal{F}$ with $|E_m| = m$. It follows that
\begin{equation}\label{e200}
\lim_{m\to \infty}\frac{\|1_{B_m}\|}{\|1_{E_m}\|}\ =\ \lim_{m\rightarrow\infty} \frac{m^{1/p}}{m^{1/2}}\ =\ \infty,
\end{equation}
so $\BB$ is not democratic.

To complete the proof of (1), we need to verify the conditionality of $\BB$. To that end, for each $m\geqslant 2$, let 
$$
x_m\ :=\ \sum_{j=1}^{m}\frac{e_j}{j^{1/2}} \mbox{ and } y_m\ :=\ \sum_{j=1}^{m}(-1)^{j}\frac{e_j}{j^{1/2}}. 
$$
We have 
$$
\|x_m\|_{\triangleleft}\ \approx\ \log(m), \|y_m\|_{\triangleleft}\ =\ 1, \mbox{ and }\|y_m\|_{\ell_2}\ \approx\ \log^{1/2}(m). 
$$
We are done by  showing that $(\|y_m\|_{\circ})_{m\in \mathbb{N}}$ is uniformly bounded. Indeed, it follows from \eqref{biggaps} that for each $m\in \mathbb{N}$,
\begin{align*}
\|y_m\|_{\circ}&\ \leqslant\  \sum_{k\in \NN}\sum_{j=1}^{k}\frac{1}{(\min A_k)^{1/2}}\left(\frac{k(k-1)}{2}+j\right)^{-1/q}\\
&\ \leqslant\ \sum_{k\in \mathbb{N}}\frac{k}{(\min A_k)^{1/2}}\ <\ \sum_{k\in \mathbb{N}}\frac{k}{b^{1/2}_{n_{k-1}+1}}\\
&\ =\ \frac{1}{b^{1/2}_2}+\sum_{k=1}^\infty \frac{k+1}{b^{1/2}_{n_{k}+1}}\ <\ \frac{1}{b^{1/2}_2}+\sum_{k=1}^\infty \frac{k+1}{\sqrt{2}k^3} \ <\ \infty.
\end{align*}

(2) This is proven with a simplification of the basis in (1). Indeed, if one removes the norm $\|\cdot\|_{\triangleleft}$ from the construction, it is easy to check that $\BB$ is unconditional, and the relevant part of the proof of (1) also proves \eqref{democracy} in this case. 
\end{proof}

\section{On $\mathcal{F}_{(a_n)}$-strong partially greedy bases}

Partially greedy Schauder bases were introduced by Dilworth et al. \cite{DKKT2003} to compare the efficiency of the TGA to that of partial summations $\sum_{n=1}^m e_n^*(x)e_n$. Later, Berasategui et al. \cite{BBL2021} extended the notion to general bases, called strong partially greedy bases. The main goals of this section is to establish the analog of Theorems \ref{bcal2} and \ref{m2} for $\mathcal{F}$-strong partially greedy bases.

\begin{defi}\normalfont (\cite[(1.4)]{DKKT2003} and \cite[Definition 1.1]{BBL2021})
A basis of a $p$-Banach space is said to be strong partially greedy if there exists $C > 0$ such that 
$$\|x-G_m(x)\|\ \leqslant\ C\inf_{k\leqslant m}\|x-S_k(x)\|, \mbox{ for all }x\in X, m\in \mathbb{N}, \mbox{ and }G_m(x).$$
\end{defi}

We can characterize strong partially greedy bases using the conservative property. 
\begin{defi}\normalfont
A basis is conservative if there is $C > 0$ such that $\|1_A\|\leqslant C\|1_B\|$ whenever $A, B\in \mathbb{N}^{<\infty}$ with $|A|\leqslant |B|$ and $A < B$.
\end{defi}

\begin{thm}(\cite[Theorem 3.4]{DKKT2003} and \cite[Theorem 4.2]{B2})\label{charPG}
A basis is strong partially greedy if and only if it is quasi-greedy and conservative. 
\end{thm}

We now define $\mathcal{F}$-strong partially greedy bases in the same spirit as $\mathcal{F}$-greedy and $\mathcal{F}$-almost greedy bases in \cite{BC}. Again, we can and shall assume that $\emptyset\in \mathcal{F}$; see Lemma \ref{pempty} for details.
\begin{defi}\normalfont\label{definitionFpartiallygreedy2}
A basis is said to be $\mathcal{F}$-strong partially greedy if there exists $\Delta\geqslant 1$ such that
\begin{equation}\label{e100}\|x-P_{\Lambda_m(x)}(x)\| \ \leqslant\ \Delta\inf_{\substack{F\in \mathcal{F}, |F|\leqslant m\\    F\setminus \Lambda_m(x) < \Lambda_m(x)\setminus F}}\|x-P_{F}(x)\|, \mbox{ for all } x\in X, m\geqslant 0, \mbox{ and } \Lambda_m(x).\end{equation}
\end{defi}

\begin{rek}\rm Note that if we take $\cF$ to be the family of all initial segments, we always have $F\setminus \Lambda_m\left(x\right)<\Lambda_m\left(x\right)\setminus F$ in \eqref{e100}, so a basis is $\cF$-strong partially greedy if and only if it is strong partially greedy, i.e., the former notion generalizes the latter. 
\end{rek}

The first order of business is to characterize $\mathcal{F}$-strong partially greedy bases using the $\mathcal{F}$-strong disjoint conservative property.

\begin{defi}\label{Fconservative2}\normalfont A  basis $\mathcal{B}$ is $\mathcal{F}$-strong disjoint conservative if there is $C > 0$ such that $\|1_A\|\leqslant C\|1_B\|$ for any $A, B\in \mathbb{N}^{<\infty}$ such that $|A|\leqslant |B|$, and there is $F\in \cF$ such that $A\subset F$, $A < B$, and $F\cap B = \emptyset$. 

If we can replace $1_A$ and $1_B$ with $1_{\varepsilon,A}$ and $1_{\delta, B}$ for any signs $\varepsilon, \delta$, we say that $\BB$ is $\mathcal{F}$-strong disjoint superconservative. 
\end{defi}

\begin{thm}\label{charFPG}
Let $\mathcal{F}$ be a family and $\BB$ be a basis of a $p$-Banach space $X$. The following are equivalent
\begin{enumerate}[\rm i)]
\item \label{FSPG}$\BB$ is $\cF$-strong partially greedy. 
\item \label{QGFSCS}$\BB$ is quasi-greedy and  $\mathcal{F}$-strong disjoint superconservative. 
\item \label{QGFSCD} $\BB$ is quasi-greedy and  $\mathcal{F}$-strong disjoint conservative. 
\end{enumerate}
\end{thm}

\begin{proof}
We first show that \ref{FSPG} implies \ref{QGFSCS}. Suppose that $\BB$ is $\Delta$-$\cF$-strong partially greedy. By choosing $F = \emptyset$ in \eqref{e100}, we see that $\BB$ is quasi-greedy.  To see that it is $\cF$-strong disjoint superconservative, fix $F\in \cF$, $A\subset F$, and $B\in \NN^{<\infty}$ so that  $|A| \leqslant |B|$, $F\cap B = \emptyset$, and $A < B$. Fix signs $\varepsilon$, $\eta$. We have 
\begin{align*}
\|1_{\varepsilon, A}\|&\ =\ \|1_{\varepsilon, A}+1_{F\setminus A}+1_{\eta, B}-(1_{F\setminus A}+1_{\eta, B})\|\\
&\ \leqslant\  \Delta  \|1_{\varepsilon, A}+1_{F\setminus A}+1_{\eta, B}-P_F(1_{\varepsilon, A}+1_{F\setminus A}+1_{\eta, B})\|=\Delta\|1_{\eta,B}\|. 
\end{align*}

That \ref{QGFSDS} implies \ref{QGFSDD} is immediate. 

Finally, to see that \ref{QGFSDD} implies \ref{FAG}, we let $C$ be a constant for which \eqref{tqg+} holds, and suppose that $\mathcal{B}$ is $C_q$-suppression quasi-greedy and $D_{c}$-$\cF$-strong disjoint conservative. Fix $x\in X$, $m\in \NN$, $\Lambda_m(x)$, and $F\in \cF$ with $|F|\leqslant m$ and $F\backslash \Lambda_m(x) < \Lambda_m(x)\backslash F$. We have
\begin{align*}
\|x-P_{\Lambda_m(x)}(x)\|^p&\ =\ \|x-P_{F\cap \Lambda_m(x)}(x)- P_{\Lambda_m(x)\setminus F}(x-P_{F\cap \Lambda_m(x)}(x))  \|^p\\
&\ \leqslant\ C_q^p\|x-P_{F\cap \Lambda_m(x)}(x)\|^p\\
&\ \leqslant\ C_q^p\|x-P_F(x)\|^p+C_q^p\|P_{F\setminus \Lambda_m(x)}(x)\|^p\\
& \ \leqslant\ C_q^p\|x-P_F(x)\|^p+C_q^pC^p\max_{n\in F\setminus \Lambda_m(x)}|e_n^*(x)|^p\|1_{F\setminus \Lambda_m(x)}\|^p\\
& \leqslant\  C_q^p\|x-P_F(x)\|^p+C_q^pC^pD_c^p\min_{n\in  \Lambda_m(x)\setminus F}|e_n^*(x)|^p\|1_{ \Lambda_m(x)\setminus F }\|^p\\
& \leqslant\  C_q^p(1+D_c^pC^{2p})\|x-P_F(x)\|^p. 
\end{align*}
Hence, $\mathcal{B}$ is $\cF$-strong partially greedy with a constant at most $C_q(1+D_c^pC^{2p})^{1/p}$. 
\end{proof}

Next, we prove Theorem \ref{m3} through the following results. 

\begin{thm}
If $(a_n)_n$ is unbounded, there exists a basis that is $\mathcal{F}_{(a_n)}$-strong partially greedy but not strong partially greedy. 
\end{thm}

\begin{proof}
Note that in the proof of Theorem \ref{propunbounded}, one can require the additional requirement on the set $E_m$ that $E_m>B_m$. Then \eqref{e200} gives that $\BB$ is not conservative and thus, not strong partially greedy according to Theorem \ref{charPG}. 
\end{proof}

\begin{thm}\label{thm: fspg->spg}
If $(a_n)_n$ is bounded, then a basis is $\mathcal{F}_{(a_n)}$-strong partially greedy if and only if it is strong partially greedy. 
\end{thm}

\begin{proof}
If a basis is strong partially greedy, then by Theorem \ref{charPG}, it is quasi-greedy and conservative. By Definition \ref{Fconservative2}, a conservative basis is $\mathcal{F}_{(a_n)}$-strong disjoint conservative. Finally, we use Theorem \ref{charFPG} to conclude that the basis is $\mathcal{F}_{(a_n)}$-strong partially greedy.

Next, assume that $\mathcal{B}$ is $\Delta$-$\mathcal{F}_{(a_n)}$-strong partially greedy.
By Theorem \ref{charPG}, we need to show that $\mathcal{B}$ is quasi-greedy and conservative. The former is obvious by replacing $F$ in \eqref{e100} with the empty set to have the $\Delta$-suppression quasi-greedy property. We prove the latter. Pick nonempty $A, B\in \mathbb{N}^{<\infty}$ with $|A|\leqslant |B|$ and $A < B$. Let $M = \max_n a_n$. Write $A = A_1\sqcup A_2$, where $A_1 = A_{< M}$ and $A_2 = A_{\geqslant M}$. First, assume that $|A| \geqslant M$ and thus, $A_2\neq \emptyset$. Choose $k_0\in \mathbb{N}$ such that 
$$(\min A_2 - a_1)+\sum_{i=1}^{k_0}a_i\ \leqslant\ \max A_2,\quad (\min A_2 - a_1)+\sum_{i=1}^{k_0+1}a_i\ >\ \max A_2.$$
Let $F = \{a_1, \ldots, a_1+\cdots+a_{k_0}\}\in \mathcal{F}_{(a_n)}$. For $1\leqslant i\leqslant M$, let
$$S_i \ =\ (i-1)+(\min A_2 - a_1)+F\ \in\ \mathcal{F}_{(a_n)}.$$

\begin{claim}\label{clcon}
The interval $[\min A_2, \max A_2-M]$ is contained in $\cup_{i=1}^M S_i$.
\end{claim}

\begin{proof}
Suppose, for a contradiction, that there exists $j$ such that $\min A_2\leqslant j\leqslant \max A_2-M$ and $j\notin \cup_{i=1}^M S_i$. Note that $j\geqslant \min A_2 + M$ because the interval $[\min A_2, \min A_2+M-1]\subset \cup_{i=1}^M S_i$. 

Clearly, $\ell_i := j-(i-1)-(\min A_2-a_1)\notin F$ for all $1\leqslant i\leqslant M$. Furthermore,
\begin{align*}
\ell_i&\ \geqslant\ (\min A_2 + M) - (M-1)-(\min A_2-a_1)\ =\ a_1 + 1,\\
\ell_i&\ \leqslant\ (\max A_2 - M)-(\min A_2-a_1)\ =\ a_1 + \max A_2 - M- \min A_2. 
\end{align*}
Hence, the $M$ consecutive numbers $\ell_k$ satisfy
$$a_1+1\ \leqslant\ \ell_k\ \leqslant\ a_1+\max A_2 - M- \min A_2\ \leqslant\ \sum_{i=1}^{k_0+1}a_i-1.$$
However, $\ell_k\notin F$ for all $k$. This contradicts that $\max a_n = M$. 
\end{proof}

We have 
\begin{equation}\label{e102}|\{n\, :\, n> A, n\in \cup_{i=1}^{M}S_i\}|\ <\ M,\end{equation}
because 
$$\max \cup_{i=1}^M S_i \ =\ \max S_M \ =\  (M-1)+(\min A_2 - a_1) + \sum_{i=1}^{k_0}a_i\ \leqslant\ (M-1)+\max A.$$
Hence, $|B\cap \left(\cup_{i=1}^{M}S_i\right)| < M$ and $|B\backslash \left(\cup_{i=1}^{M}S_i\right)|\geqslant 1$. Form $B_1 = (B\backslash \left(\cup_{i=1}^{M}S_i\right))\sqcup D$, where $D > B\cup \left(\cup_{i=1}^{M}S_i\right)$ and $|D| = M-1$. Then $|B_1|\geqslant |B|$. We define $A_{2,j}$ recursively as follows: let $A_{2,1} = A_2\cap S_1$. Then for $2\leqslant j\leqslant M$, 
$$A_{2,j} \ =\ (A_2\cap S_j)\backslash \cup_{i=1}^{j-1}A_{2,i}.$$
By the $\mathcal{F}_{(a_n)}$-strong partially greedy property, we have, for all $1\leqslant i\leqslant M$, 
\begin{align}
    \|1_{A_{2,i}}\| &\ =\ \|1_{A_{2,i}}+1_{S_i\backslash A_{2,i}}+1_{B_1}-(1_{S_i\backslash A_{2,i}}+1_{B_1})\|\nonumber\\
    &\label{e104}\ \leqslant\ \Delta\|1_{A_{2,i}}+1_{S_i\backslash A_{2,i}}+1_{B_1}-P_{S_i}(1_{A_{2,i}}+1_{S_i\backslash A_{2,i}}+1_{B_1})\|\ =\ \Delta\|1_{B_1}\|, 
\end{align}
where we can project onto $S_i$ because
\begin{itemize}
    \item $S_i\in \mathcal{F}_{(a_n)}$,
    \item $|S_i|\le |B| + |S_i|-|A_{2,i}| = |B| + |S_i\backslash A_{2, i}|  \le |B_1| +|S_i\backslash A_{2,i}|$, and
    \item $S_i\backslash ((S_i\backslash A_{2,i})\cup B_1) = A_{2, i} < B_1 =  ((S_i\backslash A_{2,i})\cup B_1)\backslash S_i$.
\end{itemize}

On the other hand, since $\mathcal{B}$ is $\Delta$-suppression quasi-greedy, 
\begin{align}
\|1_{B_1}\|^p &\ \leqslant\ \|1_{B\backslash \left(\cup_{i=1}^{M}S_i\right)}\|^p + \|1_{D}\|^p\nonumber\\
&\label{e105}\ \leqslant\ \Delta^p\|1_{B}\|^p + (M-1)c_2^{2p}\|1_B\|^p\ =\ (\Delta^p+(M-1)c_2^{2p})\|1_B\|^p.
\end{align}
From \eqref{e104} and \eqref{e105},
$$\|1_{A_{2,i}}\|\ \leqslant\ \Delta(\Delta^p+(M-1)c_2^{2p})^{1/p}\|1_B\|, \mbox{ for all }1\leqslant i\leqslant M.$$
By Claim \ref{clcon}, we can write $A = \sqcup_{i=1}^M A_{2,i}\sqcup A_1\sqcup A_0$, where $|A_0|\leqslant M$. Therefore,
\begin{align*}
\|1_A\|^p&\ \leqslant\ \sum_{i=1}^M \|1_{A_{2,i}}\|^p + \|1_{A_0\sqcup A_1}\|^p\\
&\ \leqslant\ M\Delta^p(\Delta^p+(M-1)c_2^{2p})\|1_B\|^p + 2Mc_2^{2p}\|1_B\|^p\\
&\ =\ M(\Delta^p(\Delta^p+(M-1)c_2^{2p})+2c_2^{2p})\|1_B\|^p. 
\end{align*}
It remains to consider the case $|A|\le M$. In that case, $\|1_A\|\le M^{1/p}c_2^{2}\|1_{B}\|$. We conclude that $\mathcal{B}$ is conservative, which completes our proof. 
\end{proof}

\section{Between the almost greedy and strong partially greedy properties}\label{mpg}

The following definition is inspired by an interval characterization of strong partially greedy bases in \cite[Theorem 1.2]{Ch2}. Thanks to Lemma \ref{memptyset}, we can and will assume that $\emptyset\in \mathcal{F}$.

\begin{defi}\normalfont\label{definitionFPG1}
A basis is said to be $\mathcal{F}$-minimum partially greedy if there exists $\Delta\geqslant 1$ such that
\begin{equation}\label{newdefi}\|x-P_{\Lambda_m(x)}(x)\| \ \leqslant\ \Delta\inf_{\substack{F\in \mathcal{F}, |F|\leqslant m\\  F = \emptyset \vee \min F\leqslant \Lambda_m(x)}}\|x-P_{F}(x)\|, \mbox{ for all } x\in X, m\in \mathbb{N}, \mbox{ and } \Lambda_m(x).\end{equation}
\end{defi}

\begin{rek}\rm Note that if we take $\cF$ to be the family of all initial segments, we always have $\min \left(F\right)\le \Lambda_m\left(x\right)$ in \eqref{newdefi}, so a basis is $\cF$-strong partially greedy if and only if it is strong partially greedy. This shows that the notion of $\cF$-minimum-partially greedy bases is (another) extension of the notion of strong partially greedy bases (in addition to $\cF$-strong partially greedy bases). 
\end{rek}

\begin{prop}\label{p300}
Given a family $\mathcal{F}$, if a basis is $\cF$-minimum partially greedy, then it is $\cF$-strong partially greedy.
\end{prop}

\begin{proof}
Assume that a basis satisfies \eqref{newdefi}. We shall show that the basis satisfies \eqref{e100}. Pick $x\in X, m\in \mathbb{N}$, $\Lambda_m(x)$, and $F\in \mathcal{F}$ such that $|F|\leqslant m$ and $F\backslash \Lambda_m(x) < \Lambda_m(x)\backslash F$. If $F = \emptyset$, then $\|x-P_{\Lambda_m(x)}\|\ \leqslant\ \Delta\|x-P_F(x)\|$ by \eqref{newdefi}. Suppose that $F\neq \emptyset$. 

Case 1: $F\backslash \Lambda_m(x) \neq \emptyset$; equivalently, $F\not\subset \Lambda_m(x)$.
Write $\Lambda_m(x) = (\Lambda_m(x)\backslash F)\cup (\Lambda_m(x)\cap F)$. Clearly, 
$$\min F\ \leqslant\ F\backslash \Lambda_m(x)\ <\ \Lambda_m(x)\backslash F\mbox{ and }\min F\ \leqslant\ \Lambda_m(x)\cap F.$$
Hence, $\min F\leqslant \Lambda_m(x)$. By \eqref{newdefi}, we have
$$\|x-P_{\Lambda_m(x)}(x)\|\ \leqslant\ \Delta\|x-P_F(x)\|.$$

Case 2: $F\subset \Lambda_m(x)$. Then $\Lambda_m(x)\backslash F$ is a greedy set of $x-P_F(x)$. Observe that \eqref{newdefi} implies that the basis is $\Delta$-suppression quasi-greedy. Therefore,

$$\|x-P_{\Lambda_m(x)}(x)\|\ =\ \|(x-P_F(x))-P_{\Lambda_m(x)\backslash F}(x-P_F(x))\|\ \leqslant\ \Delta\|x-P_F(x)\|.$$

This completes our proof. 
\end{proof}
As we shall see, the reverse implication of Proposition \ref{p300} does not hold in general.  In fact, depending on the sequence $(a_n)$, the $\mathcal{F}_{(a_n)}$-minimum partially greedy property may be equivalent to strong partial greediness, to almost greediness, or to properties that lie strictly between the two.  We begin our study of this property with a result analogous to Theorems~\ref{bcal2} and~\ref{charFPG}. 

\begin{defi}\label{Fconservative3}\normalfont A  basis $\mathcal{B}$ is $\mathcal{F}$-minimum disjoint conservative if there is $C > 0$ such that $\|1_A\|\leqslant C\|1_B\|$ for any $A, B\in \mathbb{N}^{<\infty}$ with $|A|\leqslant |B|$, and there is $F\in \cF$ such that $A\subset F$,  $B\cap F=\emptyset$, and $\min F< B$. 

If we can replace $1_A$ and $1_B$ with $1_{\varepsilon,A}$ and $1_{\delta, B}$ for any signs $\varepsilon, \delta$, we say that $\BB$ is $\mathcal{F}$-minimum disjoint superconservative. 
\end{defi}

\begin{thm}\label{bcal4}
Let $\mathcal{F}$ be a family and $\mathcal{B}$ be a basis of a $p$-Banach space $X$. The following are equivalent
\begin{enumerate}[\rm i)]
\item \label{FMPG}$\BB$ is $\cF$-minimum partially greedy. 
\item \label{QGFMDSC}$\BB$ is quasi-greedy and  $\mathcal{F}$-minimum disjoint superconservative. 
\item \label{QGFMDC} $\BB$ is quasi-greedy and  $\mathcal{F}$-minimum disjoint conservative. 
\end{enumerate}
\end{thm}
\begin{proof}
Let us show that \ref{FMPG} implies \ref{QGFMDSC}. Suppose that $\BB$ is $\Delta$-$\cF$-minimum partially greedy. By definition, it is quasi-greedy.  To see that it is also $\Delta$-$\cF$-minimum disjoint superconservative, fix $F\in \cF$, $A\subset F$, and $B\in \NN^{<\infty}$ so that  $F\cap B=\emptyset$, $\min F < B$, and $|A|\leqslant |B|$. Fix signs $\varepsilon$, $\eta$. We have
\begin{align*}
\|1_{\varepsilon, A}\|&\ =\ \|1_{\varepsilon, A}+1_{F\setminus A}+1_{\eta, B}-(1_{F\setminus A}+1_{\eta, B})\|\\
&\ \leqslant\  \Delta  \|1_{\varepsilon, A}+1_{F\setminus A}+1_{\eta, B}-P_F(1_{\varepsilon, A}+1_{F\setminus A}+1_{\eta, B})\|=\Delta\|1_{\eta,B}\|. 
\end{align*}

That \ref{QGFMDSC} implies \ref{QGFMDC} is immediate. 

Finally, to see that \ref{QGFMDC} implies \ref{FMPG}, we let $C$ be a constant for which \eqref{tqg+} holds, and suppose that $\mathcal{B}$ is $C_q$-suppression quasi-greedy and $C_1$-$\cF$-minimum disjoint conservative. Fix $x\in X$, $m\in \NN$, $\Lambda_m(x)$, and $F\in \cF$ with $|F|\leqslant m$ and $\min(F)\leqslant \Lambda_m(x)$. If $\Lambda_m(x)=F$, there is nothing to prove. Otherwise, 
\begin{align*}
\|x-P_{\Lambda_m(x)}(x)\|^p&\ =\ \|x-P_{F\cap \Lambda_m(x)}(x)- P_{\Lambda_m(x)\setminus F}(x-P_{F\cap \Lambda_m(x)}(x))  \|^p\\
&\ \leqslant\ C_q^p\|x-P_{F\cap \Lambda_m(x)}(x)\|^p\\
&\ \leqslant\ C_q^p\|x-P_F(x)\|^p+C_q^p\|P_{F\setminus \Lambda_m(x)}(x)\|^p\\
& \ \leqslant\ C_q^p\|x-P_F(x)\|^p+C_q^pC^p\max_{n\in F\setminus \Lambda_m(x)}|e_n^*(x)|^p\|1_{F\setminus \Lambda_m(x)}\|^p\\
& \leqslant\  C_q^p\|x-P_F(x)\|^p+C_q^pC^pC_1^p\min_{n\in  \Lambda_m(x)\setminus F}|e_n^*(x)|^p\|1_{ \Lambda_m(x)\setminus F }\|^p\\
& \leqslant\  C_q^p(1+C_1^pC^{2p})\|x-P_F(x)\|^p. 
\end{align*}
Taking infimum, we conclude that $\mathcal{B}$ is $\cF$-minimum partially greedy with the constant at most $C_q(1+C_1^pC^{2p})^{1/p}$. 
\end{proof}

We turn now to families of sets of natural numbers generated by sequences. To simplify the notation, given $(a_n)_{n\in \NN}$, $j\in \NN_0$, $\ell\in \NN$, we set
\begin{equation}\label{e401}
F_{j,\ell, (a_n)}\ :=\ j+\left\{ a_1, a_1+a_2,\dots ,\sum_{k=1}^{\ell}a_k\right\}. 
\end{equation}
If there is no ambiguity, we may just write $F_{j,\ell}$. Also, we set $r_m = \max F_{0,m}$ and $I_{m}:=\{1,\dots,m\}$ for each $m\in \NN$. Let us state a lemma that follows at once from Theorem~\ref{thm: fspg->spg} and Proposition \ref{p300}. 
\begin{lem}\label{lemmampg->spg}If $(a_n)$ is bounded and $\BB$ is $\cF_{(a_n)}$-minimum partially greedy, it is strong partially greedy. 
\end{lem}
The converse of Lemma~\ref{lemmampg->spg} does not hold in general. More precisely, we have the following result.

%
%
%
%
%

\begin{lem}\label{lemmastrong}Let $(a_n)_{n\in \NN}\subset \NN$ be a sequence. The following are equivalent 
\begin{enumerate}
\item \label{stronganstrong} Every strong partially greedy basis is $\cF_{(a_n)}$-minimum partially greedy.
\item \label{finitegaps} The set $\{n\in \NN: a_n\geqslant 2\}$ is finite. 
\end{enumerate}
Therefore, even when $(a_n)_n$ is bounded, a strong partially greedy basis may not be $\cF_{(a_n)}$-minimum partially greedy.
\end{lem}

\begin{proof}
We show that \eqref{stronganstrong} implies \eqref{finitegaps} by supposing that \eqref{finitegaps} does not hold and constructing a strong partially greedy basis that is not $\cF_{(a_n)}$-minimum partially greedy. To this end, first let $(p_k)_{k\in \NN}$ be a strictly decreasing sequence of positive real numbers greater than  $1$.  Now choose a sequence of sets $F_{1,n_j}$ (defined in \eqref{e401}) as follows. Fix $k_1>1$. Since \eqref{finitegaps} does not hold, there is $n_1>1$ such that  $\left| F_{1,n_1}\setminus (a_1+ I_{n_1})\right| \geqslant k_1$. 
Now choose $k_2 > k_1$ and $n_2 > \max F_{1,n_1}$, for which the following hold
$$
|F_{1, n_2}\setminus (a_1+I_{n_2})|\ \geqslant\ k_2\mbox{ and } k_2^{1/p_2} \ >\ 2 k_2^{1/p_1}.
$$
Recursively, choose strictly increasing sequences of positive integers $(n_j)_{j\in \NN}$, $(k_j)_{j\in \NN}$ so that for every $j\geqslant 2$, 
\begin{align*}
&|F_{1,n_j}\setminus (a_1+ I_{n_j})|\ \geqslant\ k_j;\\
&k_j^{1/p_j}\ >\ j k_j^{1/p_{j-1}}; \\
&n_j\ >\ \max F_{1,n_{j-1}}. 
\end{align*}
For each $j\in \NN$, define a seminorm on $c_{00}$ by
$$
\|(x_n)_n\|_{\triangleleft, j}\ := \left\| \left(x_n\right)_{n>a_1+n_j}\right\|_{\ell_{p_j}},
$$
and define $X$ as the completion of $c_{00}$ with the norm 
$$
\|x\|\ =\ \max\left\{ \|x\|_{\ell_{\infty}},\sup_{j\in \NN}\|x\|_{\triangleleft, j}\right\}. 
$$
Let $\BB$ be the canonical unit vector basis of $X$. It is routine to check that $\BB$ is $1$-unconditional and  $1$-conservative,  so by Theorem~\ref{charPG}, it is strong partially greedy. Suppose that $\BB$ is $C$-$\cF_{(a_n)}$-minimum partially greedy for some $C>0$. For each $j\in \NN$, let $m_j:=|F_{1,n_j}\setminus (a_1+ I_{n_j})|$ (note that $m_j=|(a_1+ I_{n_j})\setminus F_{1,n_j}|$). Given that $(p_j)_{j\in \NN}$ is decreasing and, for every $j$, we have $\min\left(F_{1,n_j}\right)=1+a_1\le a_1+ I_{n_j}$, it follows that 
\begin{align*}
m_j^{1/p_j}&\ =\ \|1_{F_{1,n_j}\setminus (a_1 + I_{n_j})}\|_{\triangleleft,j} \ \leqslant\ \|1_{F_{1,n_j}\setminus (a_1+ I_{n_j})}\|\\
&\ =\ \|1_{F_{1,n_j}\setminus (a_1+ I_{n_j})}+1_{a_1+ I_{n_j}}-1_{a_1+ I_{n_j}}\|\\
&\ \leqslant\ C\|1_{F_{1,n_j}\setminus (a_1+ I_{n_j})}+1_{a_1+ I_{n_j}}-1_{F_{1,n_j}}\|\\
&\ =\ C\|1_{(a_1+ I_{n_j})\setminus F_{1,n_j}}\|\ =\ C\sup_{l\in \NN}  \left\vert  \left(\left(a_1+I_{n_j}\right)\setminus F_{1,n_j}\right)\cap \NN_{>a_1+n_l} \right\vert^{\frac{1}{p_l}}\\
&\ =\ C\max_{1\le l\le j-1}  \left\vert  \left(\left(a_1+I_{n_j}\right)\setminus F_{1,n_j}\right)\cap \NN_{>a_1+n_l} \right\vert^{\frac{1}{p_l}}\\
&\ \le\  C\max_{1\le l\le j-1}  \left\vert  \left(a_1+I_{n_j}\right)\setminus F_{1,n_j} \right\vert^{\frac{1}{p_l}}\\
&\ =\ Cm_j^{1/p_{j-1}}.
\end{align*}
As $k_j\leqslant m_j$ for all $j\in \NN$, we have
$$
j\ <\ k_j^{1/p_j-1/p_{j-1}}\ \leqslant\ m_j^{1/p_j-1/p_{j-1}}\ \leqslant\ C, \mbox{ for all } j\in \NN_{\geqslant 2}, 
$$
a contradiction. 

Next,  we show that \eqref{finitegaps} implies \eqref{stronganstrong}. Suppose that $\BB$ is strong partially greedy. By Theorem~\ref{charPG}, it is $C_q$-quasi-greedy and $\Delta$-conservative for some positive constants $C_q$ and $\Delta$. By Theorem~\ref{bcal4}, it suffices to prove that $\BB$ is $\cF_{(a_n)}$-minimum disjoint conservative. To this end, let
$$
M_1\ =\ \left| \{n\in \NN\, :\, a_n\geqslant 2\}\right|\mbox{ and } M_2 \ :=\ \max_{n\in \NN}a_n,
$$
and pick $A, B\subset \NN$ and $j, \ell$ so that 
$$
|A|\ \leqslant\ |B|, A \ \subset\ F\ :=\ F_{j,\ell}, \mbox{ and }\min F\ <\ B.
$$
If $|A|\leqslant 6M_1M_2$, then 
$$
\|1_{A}\|\ \leqslant\ (6M_1M_2)^{1/p} c_2^2\|1_{B}\|. 
$$
Otherwise, let $I$ be the interval of minimum length containing $F$. It follows from the hypothesis that $|I\backslash F|\leqslant M_1M_2$. From $\min F <B$, $|B|\geqslant 6M_1M_2$, and $B\cap F=\emptyset$, we deduce that
$$
| B_1\ :=\ \left\{ n\in B\, :\, n> I\right\} |\ =\ |B| - |B\cap I| \ \geqslant\ |B| - |I\backslash F|\ \geqslant\ \frac{5|B|}{6}. 
$$
Now let $\{A_j\}_{1\leqslant j\leqslant 2}$ be a partition of $A$ so that $|A_j|\leqslant 5|A|/6$ for each $j = 1, 2$. As $A_j<B_1$ for $j = 1, 2$, we have
$$
\|1_{A_j}\|\ \leqslant\ \Delta\|1_{B_1}\|\ \leqslant\ \Delta C_q\|1_{B}\|, \mbox{ }j = 1, 2,
$$
which entails that 
$$
\|1_{A}\|\ \leqslant\  2^{1/p}\Delta C_q\|1_{B}\|. 
$$
We conclude that $\BB$ is $\cF_{(a_n)}$-minimum disjoint conservative with the constant at most $\max\{(6M_1M_2)^{1/p}c_2^2,2^{1/p}\Delta C_q\}$.
\end{proof}

Lemma~\ref{lemmastrong} shows that even under mild conditions on the sequence $(a_n)_{n\in\NN}$, the $\cF_{(a_n)}$-minimum partially greedy property is strictly stronger than the strong partially greedy property. This raises several questions: how strong can the $\cF_{(a_n)}$-minimum partially greedy property be? Can it be equivalent to almost greediness? Can it lie strictly between the aforementioned properties? Our next results give sufficient conditions on when the $\cF_{(a_n)}$-minimum partially greedy property is equivalent to the almost greedy property and when it lies between the almost greedy and strong partially greedy ones. We begin with the former case. 
\begin{lem}\label{lemmampgag}Let $(a_n)_{n\in \NN}$ be a bounded sequence such that 
$$\liminf\limits_{m\rightarrow \infty}\frac{|\{1\leqslant n\leqslant m: a_n\geqslant 2\}|}{m}\ >\ 0.$$
Then an $\cF_{(a_n)}$-minimum partially greedy basis is almost greedy. In particular, the conclusion holds if 
there are only finitely many $a_n$'s that are equal to $1$. 
\end{lem}

\begin{proof}
The hypothesis implies the existence of $m_1\in \NN$ and $\alpha>0$ such that 
\begin{equation}
\left| \left\{ 1\leqslant n\leqslant m\, :\, a_n\geqslant 2\right\}\right| \ \geqslant\ \alpha\left| \left\{ 1\leqslant n\leqslant m\,:\, a_n= 1\right\}\right|, \mbox{ for all }m\geqslant m_1.\label{finiteones+}
\end{equation}
By Theorem~\ref{dkkt} and Definition~\ref{definitionFPG1}, it suffices to prove that $\BB$ is democratic. Set
$$
m_2\ :=\ \left\lceil\max\{\alpha^{-1}, m_1, \max_{n\in \NN}a_n\}\right\rceil. 
$$
Note that $m_2\geqslant 2$, and \eqref{finiteones+} holds when replacing $m_1$ and $\alpha$ with $m_2$ and $m_2^{-1}$, respectively. By Theorem~\ref{charPG},  Theorem~\ref{bcal4}, and Lemma~\ref{lemmampg->spg}, there are $C_q, \Delta_1, \Delta_2>0$ such that $\BB$ is $C_q$-quasi greedy, $\Delta_1$-conservative, and $\Delta_2$-$\mathcal{F}$-minimum disjoint conservative. Fix $A, B$ nonempty subsets of $\NN$ such that $|A|\leqslant |B| =: m$. We consider the following cases.

Case 1: If $|A|\leqslant16m_2^3$, then 
$$
\|1_{A}\|\ \leqslant\  \left(16m_2^3\right)^{1/p}c_2^2\|1_{B}\|. 
$$

Case 2:  If $|A|> 16 m_2^3$, then $m>16m_2^3$. Let 
$$
m_3\ :=\ \floor{\frac{m}{2m_2}}\mbox{ and } m_4\ :=\ \max F_{0,m_3} \ \leqslant\ m_2m_3.
$$
As $m_3>m_1$, it follows from  \eqref{finiteones+} that 
\begin{align*}
\left| \left\{ 1\leqslant n\leqslant m_3\, :\, a_n \geqslant 2\right\}\right|&\ \geqslant\ \frac{1}{m_2}\left| \left\{ 1\leqslant n\leqslant m_3\,:\, a_n= 1\right\}\right|\\
&\ =\ \frac{m_3}{m_2}-\frac{1}{m_2}\left| \left\{ 1\leqslant n\leqslant m_3\,:\, a_n \geqslant 2\right\}\right|,
\end{align*}
which implies
$$\left| \left\{ 1\leqslant n\leqslant m_3\, :\, a_n \geqslant 2\right\}\right|\ \geqslant\ \frac{m_3}{m_2+1}.$$
$$$$
Hence, 
\begin{align}
&\left| I_{m_4}\setminus F_{0,m_3}\right| \ \geqslant\ \left| \left\{ 1\leqslant n\leqslant m_3\, :\, a_n \ \geqslant\ 2\right\}\right|\ \geqslant\ \frac{m_3}{m_2+1}\ \geqslant\ \frac{m}{4m_2^2}\ >\ 4m_2\ \geqslant\ 4a_1.\label{freeinterval}
\end{align}
Let $m_5:=\floor{\frac{m}{8m_2^2}}$. By \eqref{freeinterval}, there is $B_1$ such that 
$B_1\subset I_{m_4}\setminus F_{0,m_3}$, $|B_1| = m_5$, and $a_1<B_1$. 
Since $m_4\leqslant m_2m_3\leqslant m/2$, 	we have $|B_0:=B\setminus I_{m_4}|\geqslant m/2$. As $\BB$ is $C_q$-quasi-greedy, 
\begin{equation}
\|1_{B_0}\|\ \leqslant\ C_q\|1_{B}\|.\label{QGbound}
\end{equation}
Choose $m_6>m$ sufficiently large so that there is $B_2\subset F_{0,m_6}$ with $B_2>A\cup I_{m_4}$ and $|B_2|=m_5$. As $B_1\cap F_{0,m_6}=\emptyset$, $B_2\subset F_{0,m_6}$, and $\min F_{0,m_6}=a_1<B_1$, the minimum disjoint conservative property gives
\begin{equation}\label{minimumdisjcons}
\|1_{B_2}\|\ \leqslant\ \Delta_2\|1_{B_1}\|.  
\end{equation}
Since $|A|\leqslant m$, by our choice of $m_5$, there is a partition $(A_j)_{1\leqslant j\leqslant 16m_2^2}$ of $A$ with $|A_j|\leqslant m_5$  for all $1\leqslant j\leqslant 16m_2^2$ (some $A_j$'s might be empty). For each $1\leqslant j\leqslant 16m_2^2$, the conservative property gives
$$
\|1_{A_j}\|\ \leqslant\ \Delta_1\|1_{B_2}\|,
$$
so by $p$-convexity,
\begin{equation}
\|1_{A}\|\ \leqslant\ \Delta_1(16m_2^2)^{1/p}\|1_{B_2}\|. \label{AB1}
\end{equation}
Finally, since $|B_1|=m_5<m/2\leqslant |B_0|$ and $B_1<B_0$, we have $\|1_{B_1}\|\leqslant \Delta_1\|1_{B_0}\|$, which, in combination with \eqref{QGbound}, \eqref{minimumdisjcons}, and \eqref{AB1}, yields 
\begin{equation*}
\|1_{A}\|\ \leqslant\ C_q\Delta_1^2\Delta_2(16m_2^2)^{1/p}\|1_{B}\|.
\end{equation*}
Comparing the above estimates, it follows that $\BB$ is democratic with the constant no greater than $(16m_2^2)^{1/p}\max\{m_2^{1/p}c_2^2, C_q\Delta_1^2\Delta_2\}$. 
\end{proof}

To close this section, we find sufficient conditions on $(a_n)_{n\in \NN}$ so that the $\cF_{(a_n)}$-minimum partially greedy property lies strictly between the strong partially greedy and the almost greedy properties. To do so, we shall combine examples from \cite{BDKOW2019} and  \cite{KT1999} with a construction involving sequences in our context. 
\begin{lem}\label{lemmabetween}Let $(a_n)_{n\in\NN}$ be a bounded sequence. Suppose that there is $\alpha>0$ such that for all $m\in \mathbb{N}$, 
\begin{equation}
\left| \left \{ 1\leqslant n\leqslant m\, :\, a_n\geqslant 2\right\} \right| \ \leqslant\ \alpha m^{1/4}.\label{between2}
\end{equation}
Then there is a conditional Schauder basis that is $\cF_{(a_n)}$-minimum partially greedy but is not democratic. 
\end{lem}
\begin{proof}
Let $M:=\max_{n\in \NN}a_n$ and $M_1:=\max\{M, \ceil{\alpha}\}$. Let us define seminorms on $c_{00}$ as follows: for each $\ell\in \NN$, let $r_\ell:=\max F_{0,\ell}$, and for $j\in \NN_0$ and $\ell\in \NN$, let 
$$
\left\| x=(x_n)_{n\in \NN}\right\|_{j,\ell}\ :=\ \|P_{(j+ I_{r_\ell})\setminus F_{j,\ell}}(x)\|_{\ell_2}=\left(\sum_{n\in \{j+1,\dots, j+r_\ell\}\setminus F_{j,\ell}}|x_n|^2\right)^{1/2}. 
$$
Now for each $A\in \NN^{<\infty}$ such that $|A|^2<A$, set 
\begin{equation*}
\|x\|_{A}\ :=\ \|P_A(x)\|_{\ell_2}. 
\end{equation*}
Finally, for each $m\in \NN$, set 
\begin{equation*}
\|x\|_{\circ, m}\ :=\ \max_{1\leqslant k\leqslant m}\left| \sum_{n=1}^{k}\frac{x_{m^2+n}}{n^{1/2}} \right| .
\end{equation*}
Now let 
\begin{align*}
\|x\|\ :=\ &\max\left\{ \|x\|_{\ell_{\infty}}, \sup_{\substack{A\in \NN^{<\infty}\\|A|^2<A}}\|x\|_{A},\sup_{m\in \NN}\|x\|_{\circ,m}, \sup_{j\in \NN_0,\ell\in \NN}\|x\|_{j,\ell}\right\}. 
\end{align*}
Let $\BB$ be the canonical vector basis of $X$. It is routine to check that $\BB$ is a monotone normalized Schauder basis. 

\textbf{Step 1:} To prove that the basis is quasi-greedy, we need only to bound the seminorms $\|\cdot\|_{\circ,m}$ of projections onto greedy sets. To do so, pick $x\in X$, $A\in\NN^{<\infty}$ a greedy set of $x$, and $m\in \NN$. Let $B:=A\cap (m^2+ I_{m})$ and $y:=P_{m^2+ I_{m}}(x)$.  We may assume $B\not=\emptyset$. Then $B$ is a greedy set of $y$, and 
\begin{equation*}
\|y\|_{\circ,m}\ =\ \|x\|_{\circ, m}\ \leqslant\ \|x\|\mbox{ and }\|P_A(x)\|_{\circ,m}\ =\ \|P_B(y)\|_{\circ,m}. 
\end{equation*}
Let $\cY=(\yy_n)_{n\in \NN}$ be the basis of Lemma~\ref{lemmakt}, $C_q$ be its quasi-greedy constant, and $\|\cdot\|_{\triangleleft}$ be the norm of the space. Define
$$
z\ :=\ \sum_{n=1}^m e^*_{m^2+n}(x)\yy_n\mbox{ and }D\ :=\ B-m^2. 
$$
Then $D$ is a greedy set of $z$, and 
\begin{align*}
\|P_A(x)\|_{\circ,m}&\ =\ \|P_B(y)\|_{\circ,m}\ =\ \max_{1\leqslant k\leqslant m} \left| \sum_{\substack {m^2+n\in B\\ 1\leqslant n\leqslant k}}\frac{e^*_{m^2+n}(x)}{n^{1/2}}\right|\ =\ \max_{1\leqslant k\leqslant m} \left| \sum_{\substack {n\in D\\ 1\leqslant n\leqslant k}}\frac{\yy_n^*(z)}{n^{1/2}}\right|\\
&\ \leqslant\  \left\| P_{D}(z)\right\|_{\triangleleft}\ \leqslant\ C_q\|z\|_{\triangleleft}\ =\ C_q\max\{\|y\|_{\circ,m},\|y\|_{m^2+ I_{m}}\}\\
&\ =\ C_q\max\{\|x\|_{\circ,m},\|x\|_{m^2+ I_{m}}\}\ \leqslant\ C_q\|x\|,
\end{align*}
so $\BB$ is $C_q$-quasi-greedy. 

\textbf{Step 2:} Next, we prove that $\BB$ is conditional. For $m\in \NN$, let 
$$
z_m\ :=\ \sum_{n=1}^{m}\frac{e_{m^2+n}}{n^{1/2}}\mbox{ and }u_m\ :=\ \sum_{n=1}^{m}(-1)^n \frac{e_{m^2+n}}{n^{1/2}}. 
$$
Then 
$$
\|z_m\|\ \geqslant\ \|z_m\|_{\circ,m}\ =\ \sum_{n=1}^{m}\frac{1}{n}. 
$$
On the other hand, given that the coefficients of $u_m$ alternate in signs and are decreasing, for every $\ell\in \NN$, we have $\|u_m\|_{\circ,\ell}\leqslant 1$. 
Given that all of the other seminorms are bounded by the $\ell_2$-norm, we obtain 
\begin{equation*}
\|u_m\|\ \leqslant\ \left(\sum_{n=1}^{m}\frac{1}{n}\right)^{1/2},
\end{equation*}
from which conditionality follows. 

\textbf{Step 3:} Now let us see that $\BB$ is $\cF_{(a_n)}$-minimum disjoint superconservative. We start by establishing several key inequalities for general sets $A$ and signs $\varepsilon$. First, since
 $$\sum_{n=1}^m\frac{1}{n^{1/2}}\ \leqslant\ 2m^{1/2},\mbox{ for all } m\in \NN,$$ it follows that
\begin{equation}
\|1_{\varepsilon, A}\|\ \leqslant\ 2 |A|^{1/2}, \mbox{ for all } A\in \NN^{<\infty} \mbox{ and signs } \varepsilon.\label{upperbound1A}
\end{equation}
On the other hand, 
\begin{equation}
\|1_{\varepsilon, A}\|\ \geqslant\ |A|^{1/2}, \mbox{ for all }A\in \NN^{<\infty} : |A|^2<A\mbox{ and signs }\varepsilon.\label{lowerbound1A}
\end{equation}

\begin{claim}
Given $A\in \NN^{<\infty}$ with $|A|\geqslant 4$ and any sign $\varepsilon$,
\begin{equation}\label{bound1/4}\|1_{\varepsilon, A}\|\ \geqslant\ \frac{1}{2}|A|^{1/4}.\end{equation}
\end{claim}

\begin{proof}
Write $A = A_1\sqcup A_2$, where $A_2$ consists of the largest $\ceil{(|A|-1)/2}$ elements of $A$. Then 
$$A_2 \ >\ \max A_1 \ \geqslant\ |A_1| \ =\ |A| - \ceil{\frac{|A|-1}{2}}\ \geqslant\ \frac{|A|}{2}.$$
Let $A_0$ be the set consisting of the largest $\floor {(|A|/2)^{1/2}}$ elements of $A$. Since
$$\left(\frac{|A|}{2}\right)^{1/2} \ \leqslant\ \frac{|A|-1}{2}, \mbox{ for }|A|\ \geqslant\ 4,$$
we deduce that $A_0\subset A_2$ and thus,
$$A_0 \ >\ \frac{|A|}{2} \ \geqslant\ \floor {\left(\frac{|A|}{2}\right)^{1/2}}^{2}\ =\ |A_0|^2.$$
Hence,
$$\|1_{\varepsilon, A}\|\ \geqslant\ \|1_{\varepsilon, A}\|_{A_0}\ =\ \floor {\left(\frac{|A|}{2}\right)^{1/2}}^{1/2}\ \geqslant\ \frac{1}{2}|A|^{1/4},$$
as desired.
\end{proof}

Also, note that given $A\in \NN^{<\infty}$ and a sign $\varepsilon$, for each $m\in \NN$,  we have
\begin{equation}
\left\| 1_{\varepsilon, A}\right\|_{\circ,m}\ \leqslant\ \sum_{n=1}^{|A\cap (m^2+ I_{m})|}\frac{1}{n^{1/2}}\ \leqslant\ 2 |A\cap (m^2+ I_{m})|^{1/2}\ =\ 2\|1_{\varepsilon, A}\|_{m^2+ I_{m}}. \label{boundfortheroots}
\end{equation}

We are now ready to prove the $\mathcal{F}_{(a_n)}$-minimum disjoint superconservative property. Fix $A, B\in \NN^{<\infty}$, signs $\varepsilon, \varepsilon'$, and $j_0\in \NN_0$, $\ell_0\in \NN$ such that 
$|A|\leqslant |B|$, $A\subset F_{j_0,\ell_0}\subset B^{c}$, and $j_0+a_1<B$. Let $m_2=|B|$. 

Case 1: If $m_2\leqslant 2M_1$, then 
$$
\|1_{\varepsilon, A}\|\ \leqslant\ (2M_1)^{1/p}c_2^2\|1_{\varepsilon',B}\|. 
$$

Case 2: If $m_2>2M_1$, set
$$
B_1\ :=\ \{n\in B: n>m_2^2\}\mbox{ and }B_{2}\ :=\ B\cap (j_0+ I_{r_{\ell_0}}). 
$$
If $|B_1|\geqslant m_2/4$, then  by \eqref{upperbound1A},
$$
\|1_{\varepsilon',B}\|\ \geqslant\ \|1_{\varepsilon',B}\|_{B_1}\ \geqslant\ \left(\frac{m_2}{4}\right)^{1/2}\ \geqslant\ \left(\frac{|A|}{4}\right)^{1/2}\ \geqslant\ \frac{\|1_{\varepsilon, A}\|}{4}.
$$
If $|B_2|\geqslant m_2/4$, then due to $B\cap F_{j_0,\ell_0}=\emptyset$ and \eqref{upperbound1A},
$$
\|1_{\varepsilon',B}\|\ \geqslant\ \|1_{\varepsilon',B}\|_{j_0,\ell_0}\ \geqslant\ \left(\frac{m_2}{4}\right)^{1/2}\ \geqslant\ \left(\frac{|A|}{4}\right)^{1/2}\ \geqslant\ \frac{\|1_{\varepsilon, A}\|}{4}.
$$
If $\max\{|B_1|, |B_2|\} <  m_2/4$, then $j_0+r_{\ell_0} < m_2^2$. 
Indeed, if $j_0 + r_{\ell_0} \geqslant m_2^2$, then $|B_1| + |B_2|\geqslant |B|$ and so,
$$\frac{m_2}{4}\ >\ |B_1| \ \geqslant\ |B\backslash B_2|\ =\  |B| -  |B_2|\ >\ m_2 - \frac{m_2}{4} \ =\ \frac{3m_2}{4},$$
a contradiction. That $j_0 + r_{\ell_0} \leqslant m_2^2$ implies that $A\subset I_{m_2^2}$. 

Pick $j\in \mathbb{N}_0$, $\ell\in \mathbb{N}$. We have 

\begin{align*}
\left| A\cap ((j+ I_{r_\ell})\setminus  F_{j,\ell})\right|&\ \leqslant\ \left| I_{m^2_2}\cap ((j + I_{r_\ell})\setminus  F_{j,\ell})\right|\\
&\ \leqslant\ \left| I_{m^2_2}\cap (I_{r_\ell}\setminus  F_{0,\ell})\right|.
\end{align*}
\begin{enumerate}
\item Case 2.1: If $m_2^2 \geqslant r_\ell$, then by \eqref{between2},
\begin{align*}
    \left| A\cap ((j+ I_{r_\ell})\setminus  F_{j,\ell})\right|&\ \leqslant\ \left| I_{r_\ell}\setminus  F_{0,\ell}\right|\\
    &\ \leqslant\ M|\{1\leqslant n\leqslant \ell: a_n\geqslant 2\}|\\
    &\ \leqslant\ MM_1 r_\ell^{1/4}\ \leqslant\ M_1^2 m_2^{1/2}.
\end{align*}

\item Case 2.2: If $m_2^2 < r_\ell$, then let $\ell_s$ be the largest such that $r_{\ell_s}\leqslant m_2^2$ to have
\begin{align*}
    \left| A\cap ((j+ I_{r_\ell})\setminus  F_{j,\ell})\right|&\ \leqslant\ \left| I_{m_2^2}\setminus  F_{0,\ell}\right|\\
    &\ \leqslant\ M|\{1\leqslant n\leqslant \ell_s: a_n\geqslant 2\}| + M\\
    &\ \leqslant\ M(M_1r^{1/4}_{\ell_s} + 1)\ \leqslant\ 2M_1^2m_2^{1/2}.
\end{align*}
\end{enumerate}
Hence, for all $j\in \mathbb{N}_0$, $\ell\in \mathbb{N}$, \eqref{upperbound1A} and \eqref{bound1/4} give
\begin{equation}
\left\| 1_{\varepsilon, A}\right\|_{j,\ell}\ \leqslant\ (4M_1^2m_2^{1/2})^{1/2}\ \leqslant\ 2M_1 m_2^{1/4}\ \leqslant\ 4M_1\|1_{\varepsilon',B}\|. \label{boundjl}
\end{equation}

Now let $$B_3\ :=\ \{n\in B\, :\, j_0+r_{\ell_0} < n\leqslant m_2^2\}.$$
Note that $B_3=B\setminus (B_1\cup B_2)$. Since $\max\{|B_1|, |B_2|\}\leqslant m_2/4$, it follows that $|B_3|\geqslant m_2/2$. Let $(A_j)_{1\leqslant j\leqslant 8}$ be a partition of $A$ with $|A_j|\leqslant |B_3|$ for all $1\leqslant j\leqslant 8$ (some of the $A_j$'s might be empty). Given that $A_j<B_3$ for all $1\leqslant j\leqslant 8$, we have
\begin{equation}
\sup_{\substack{D\in \NN^{<\infty}\\|D|^2<D}}\|1_{\varepsilon, A}\|_{D}\ \leqslant\ 8\max_{1\leqslant j\leqslant 8}\sup_{\substack{D\in \NN^{<\infty}\\|D|^2<D}}\|1_{\varepsilon, A_j}\|_{D} \ \leqslant\ 8\sup_{\substack{D\in \NN^{<\infty}\\|D|^2<D}}\|1_{\varepsilon', B_3}\|_{D}\ \leqslant\ 8C_q\|1_{\varepsilon',B}\|. \label{boundforcircm}
\end{equation}
 Combining \eqref{boundfortheroots}, \eqref{boundjl}, and \eqref{boundforcircm} we get 
 \begin{equation*}
 \|1_{\varepsilon,A}\|\ \leqslant\ 16C_qM_1\|1_{\varepsilon',B}\|. 
 \end{equation*}
 Therefore, $\BB$ is $16C_qM_1$-$\cF_{(a_n)}$-minimum disjoint superconservative.

\textbf{Step 4:} To complete the proof of the lemma, it remains to show that $\BB$ is not democratic. To that end, fix $m_0\in \NN_{>2M_1}$, and let us compare the norms of $1_{I_{m_0}}$ and $1_{m_0^2+I_{m_0}}$. By \eqref{lowerbound1A}, 
$$
 \|1_{m_0^2+I_{m_0}}\|\ \geqslant\  m_0^{1/2}. 
$$
Now pick $A\in \NN^{<\infty}$ so that $|A|^2<A$ and $A\cap I_{m_0}\neq \emptyset$. Then $|A|^2\leqslant m_0$, so 
$$
 \left\| 1_{I_{m_0}}\right\|_{A}\ \leqslant\ |A|^{1/2}\ \leqslant\ m_0^{1/4}. 
$$
Taking supremum and considering \eqref{boundfortheroots}, we obtain 
$$
\max\left\{ \left\| 1_{I_{m_0}}\right\|_{\ell_{\infty}}, \sup_{\substack{A\in \NN^{<\infty}\\|A|^2<A}} \left\| 1_{I_{m_0}}\right\|_{A},\sup_{m\in \NN} \left\| 1_{I_{m_0}}\right\|_{\circ,m} \right\}\ \leqslant\ 2 m_0^{1/4}.
$$
Fix $j_0\in \NN_0$ and $\ell_0\in \NN$ so that 
$$
\sup_{\substack{j\in \NN_0\\\ell\in \NN}}\left\| 1_{I_{m_0}}\right\|_{j,\ell}\ =\ \left\| 1_{I_{m_0}}\right\|_{j_0,\ell_0},
$$
and $\ell_0$ is the minimum $\ell\in \NN$ for which the supremum is attained. Since the supremum is not attained for $j_0,\ell_0-1$, it follows that $m_0\geqslant j_0+r_{\ell_0}-M$, so 
$$j_0+r_{\ell_0}\ \leqslant\ m_0+M\ \leqslant\ 2m_0.$$
Then by \eqref{between2}, 
$$
\left| I_{m_0}\cap ((j_0+ I_{r_{\ell_0}})\setminus F_{j_0,\ell_0})\right| \ \leqslant\ M \left| \left \{ 1\leqslant n\leqslant \ell_0: a_n\geqslant 2\right\} \right| \ \leqslant\ M_1^2r^{1/4}_{\ell_0};
$$
hence,
$$
\left\| 1_{I_{m_0}}\right\|_{j_0,\ell_0}\ \leqslant\ 2M_1m_0^{1/8}.
$$
Combining the above inequalities, we conclude that
$$
\frac{ \left\| 1_{I_{m_0}}\right\|}{\left\| 1_{m_0^2+ I_{m_0}}\right\|}\ \leqslant\ \frac{2M_1}{m_0^{1/4}}. 
$$
As $m_0$ is arbitrary, this proves that $\BB$ is not democratic. 
\end{proof}

\begin{rek}\label{remarkuncond}\normalfont Under the hypotheses of Lemma~\ref{lemmabetween}, we can easily substitute unconditionality for conditionality by removing the seminorms $\|x\|_{\circ, m}$ from our construction. 
\end{rek}

\begin{proof}[Proof of Theorem \ref{corollarystrict}]
Just combine Lemmas \ref{lemmampg->spg}, \ref{lemmastrong}, and \ref{lemmabetween}. 
\end{proof}

\section{Appendix}

Given a family $\mathcal{F}$, which does not contain the empty set, let $\mathcal{A} = \{(x, m, \Lambda_m(x))\in X\times \mathbb{N}\times \mathbb{N}^{<\infty}: \exists F\in \mathcal{F}, |F|\leqslant m, F\backslash \Lambda_m(x) < \Lambda_m(x)\backslash F\}$. Consider the following condition: there exists $\Delta \geqslant 1$ such that for all $(x, m, \Lambda_m(x))\in \mathcal{A}$,
\begin{equation}\label{e300}\|x-P_{\Lambda_m(x)}(x)\|\ \leqslant\ \Delta \inf_{\substack{F\in \mathcal{F}, |F|\leqslant m\\ F\backslash \Lambda_m(x) < \Lambda_m(x)\backslash F}}\|x-P_F(x)\|.\end{equation}
\begin{lem}\label{pempty}
The condition \eqref{e300} is equivalent to the following condition: there exists $C \geqslant 1$ such that 
\begin{equation}\label{e301}\|x-P_{\Lambda_m(x)}(x)\|\ \leqslant\ C\inf_{\substack{F\in \mathcal{F}\cup \{\emptyset\}, |F|\leqslant m\\ F\backslash \Lambda_m(x) < \Lambda_m(x)\backslash F}}\|x-P_F(x)\|, \mbox{ for all }x\in X, m\geqslant 0, \mbox{ and }\Lambda_m(x).\end{equation}
\end{lem}
\begin{proof}
That \eqref{e301} implies \eqref{e300} is immediate. Let us assume \eqref{e300} and prove \eqref{e301}.  It suffices to prove that the basis is quasi-greedy. 
Choose $G\in \mathcal{F}$ with $|G| = \min \mathcal{F}$. Pick $x\in X$, $m\in\mathbb{N}$, and $\Lambda_m(x)$.

If $m\leqslant 2\max G$, then 
\begin{align*}
\|x-P_{\Lambda_m(x)}(x)\|^p\ \leqslant\ \|x\|^p + \|P_{\Lambda_m(x)}(x)\|^p&\ \leqslant\ \|x\|^p + 2c_2^{2p}\max G\|x\|^p\\
&\ =\ (1+2c_2^{2p}\max G)\|x\|^p.
\end{align*}

Suppose that $m > 2\max G$. Write $\Lambda_1 = \Lambda_m(x)_{\leqslant \max G}$ and $\Lambda_2 = \Lambda_m(x)_{>\max G}$. Observe that $\Lambda_2$ is a greedy set of $y:= x-P_{\Lambda_1}(x)$. Furthermore, 
$$|\Lambda_2| \ = \ |\Lambda_m(x)| - |\Lambda_1|\ >\ 2\max G - \max G \ =\ \max G\ \geqslant\ |G|, \mbox{ and}$$
$$G\backslash \Lambda_2 \ =\ G \ <\ \Lambda_2 \ =\ \Lambda_2\backslash G.$$
By \eqref{e300},
\begin{align*}
\|x-P_{\Lambda_m(x)}(x)\|^p\ =\ \|y-P_{\Lambda_2}(y)\|^p\ \leqslant\ \Delta^p\|y-P_G(y)\|^p&\ \leqslant\ \Delta^p(\|y\|^p + \|P_G(y)\|^p)\\
&\ \leqslant\ \Delta^p(\|y\|^p + |G|c_2^{2p}\|y\|^p)\\
&\ =\ \Delta^p(1+|G|c_2^{2p})\|y\|^p.
\end{align*}
On the other hand,
$$\|y\|^p\ \leqslant\ \|x\|^p + \|P_{\Lambda_1}(x)\|^p\ \leqslant\ \|x\|^p + c_2^{2p}\max G \|x\|^p\ =\ (1+c_2^{2p}\max G)\|x\|^p.$$
Combining the above inequalities, we obtain
$$\|x-P_{\Lambda_m(x)}(x)\|^p\ \leqslant\ \Delta^p(1+|G|c_2^{2p})(1+c_2^{2p}\max G)\|x\|^p.$$
Hence, the basis is $K$-suppression quasi-greedy with
$$K\ \leqslant\ \max\{(1+2c_2^{2p}\max G)^{1/p}, \Delta(1+|G|c_2^{2p})^{1/p}(1+c_2^{2p}\max G)^{1/p}\}.$$
\end{proof}

Given a family $\mathcal{F}$, which does not contain the empty set, let $\mathcal{D} = \{(x, m, \Lambda_m(x))\in X\times \mathbb{N}\times \mathbb{N}^{<\infty}: \exists F\in \mathcal{F}, |F|\leqslant m, \min F\leqslant \Lambda_m(x)\}$. Consider the following condition: there exists $\Delta \geqslant 1$ such that for all $(x, m, \Lambda_m(x))\in \mathcal{D}$,
\begin{equation}\label{e303}\|x-P_{\Lambda_m(x)}(x)\|\ \leqslant\ \Delta \inf_{\substack{F\in \mathcal{F}, |F|\leqslant m\\ \min F \leqslant \Lambda_m(x)}}\|x-P_F(x)\|.\end{equation}

\begin{lem}\label{memptyset}
The condition \eqref{e303} is equivalent to the following condition: there exists $C \geqslant 1$ such that 
\begin{equation}\label{e304}\|x-P_{\Lambda_m(x)}(x)\|\ \leqslant\ C\inf_{\substack{F\in \mathcal{F}\cup \{\emptyset\}, |F|\leqslant m\\ F = \emptyset \vee \min F \leqslant \Lambda_m(x)}}\|x-P_F(x)\|, \mbox{ for all }x\in X, m\geqslant 0, \mbox{ and }\Lambda_m(x).\end{equation}
\end{lem}

\begin{proof}
That \eqref{e304} implies \eqref{e303} is obvious. We prove the converse by assuming \eqref{e303} and proving that the basis is quasi-greedy. Take $G\in \mathcal{F}$ with $|G| = \min \mathcal{F}$. Pick $x\in X$, $m\in \mathbb{N}$, and $\Lambda_m(x)$. If $m\leqslant \min \mathcal{F}$, then 
$$\|x-P_{\Lambda_m}(x)\|^p\ \leqslant\ \|x\|^p + \|P_{\Lambda_m(x)}(x)\|^p\ \leqslant\ (1+ c_2^{2p}\min \mathcal{F})\|x\|^p.$$
Suppose that $m > \min \mathcal{F} = |G|$. Let $s = \min G, \Lambda = \Lambda_m(x)_{\geqslant s}$, and $y:= x- S_{s-1}(x)$. Then $\Lambda$ is a greedy set of $y$. Choose a set $M$ such that $\Lambda\subset M$, $M\geqslant s$, and $M$ is a greedy set of $y$ with $|M| = m$. By \eqref{e303},
$$\|x-S_{s-1}(x)-P_{\Lambda}(x)-P_{M\backslash \Lambda}(x)\|\ =\ \|y-P_{M}(y)\|\ \leqslant\ \Delta\|y-P_G(y)\|,$$
which, together with $|M\backslash \Lambda| \leqslant s-1$, gives 
\begin{align*}
&\|x-P_{\Lambda_m(x)}(x)\|^p\\
\ \leqslant\ &\|x-P_{\Lambda_m(x)}(x)-P_{[1,s-1]\backslash \Lambda_m(x)}(x)-P_{M\backslash \Lambda}(x)\|^p + \|P_{[1,s-1]\backslash \Lambda_m(x)}(x)+P_{M\backslash \Lambda}(x)\|^p\\
\ =\ &\|x-S_{s-1}(x)-P_{\Lambda}(x)-P_{M\backslash \Lambda}(x)\|^p + 2(s-1)c_2^{2p}\|x\|^p\\
\ \leqslant\ &\Delta^p\|y-P_G(y)\|^p + 2(s-1)c_2^{2p}\|x\|^p\\
\ \leqslant\ &\Delta^p\|y\|^p + \Delta^p\|P_G(y)\|^p + 2(s-1)c_2^{2p}\|x\|^p\\
\ \leqslant\ &\Delta^p(1+c_2^{2p}\min\mathcal{F})\|y\|^p  + 2(s-1)c_2^{2p}\|x\|^p.
\end{align*}
Note that $\|y\|^p\leqslant(1+(s-1)c_2^{2p})\|x\|^p$.
Combining the above inequalities, we obtain
$$\|x-P_{\Lambda_m(x)}(x)\|\ \leqslant\ (\Delta^p(1+c_2^{2p}\min\mathcal{F})(1+(s-1)c_2^{2p}) + 2(s-1)c_2^{2p})^{1/p}\|x\|.$$
Therefore, the basis is quasi-greedy. 
\end{proof}


\ \\
\end{document}